\documentclass{article}
\usepackage[comma,square,sort&compress, numbers]{natbib}
\usepackage[greek, english]{babel}
\usepackage{amsbsy,amssymb,amscd,amsmath,amsthm,amsfonts}
\usepackage{eurosym}
\usepackage{epsf}
\usepackage{epsfig}
\usepackage{color}
 
 \def\Z{{\mathbb{Z}}} \def\R{{\mathbb{R}}} \def\N{{\mathbb{N}}} \def\Q{{\mathbb{Q}}} \def\X{{\mathbb{X}}} 
\def\p{\partial} \def\M{{\mathcal{M}}} \def\B{{\mathcal{B}}} \def\BM{\mathcal{B}\mathcal{M}}

\newtheorem{theorem}{Theorem}[section]
\newtheorem{lemma}[theorem]{Lemma}
\newtheorem{proposition}[theorem]{Proposition}
\newtheorem{corollary}[theorem]{Corollary}
\newtheorem{definition1}[theorem]{Definition}
\newenvironment{definition}{\begin{definition1}\rm}{\end{definition1}}
\newtheorem*{theorema}{Theorem A}
\newtheorem*{theoremb}{Theorem B}
\newtheorem*{theoremb1}{Theorem B1}
\newtheorem*{theoremb2}{Theorem B2}
\newtheorem*{theoremd}{Dichotomy Theorem}

\newtheorem{remark1}[theorem]{Remark}
\newenvironment{remark}{\begin{remark1}\rm}{\end{remark1}}

\newtheorem{example1}[theorem]{Example}

\def\barray{\begin{eqnarray*}}             \def\earray{\end{eqnarray*}}
\def\beq{\begin{equation}} \def\eeq{\end{equation}}

\makeatletter 
\title{A dichotomy theorem for minimizers of monotone recurrence relations}  
\author{Bla\v z Mramor\thanks{Department of Mathematics, VU University Amsterdam, The Netherlands, {\tt b.mramor@vu.nl}.} \ and Bob Rink\thanks{Department of Mathematics, VU University Amsterdam, The Netherlands, {\tt b.w.rink@vu.nl}.}\ . }
\begin{document}  
\hyphenation{ge-ne-ra-lized}
\maketitle
\noindent 

\abstract{Variational monotone recurrence relations arise in solid state physics as generalizations of the Frenkel-Kontorova model for a ferromagnetic crystal. For such problems, Aubry-Mather theory establishes the existence  of ``ground states'' or  ``global minimizers" of arbitrary rotation number.

A nearest neighbor crystal model is equivalent to a Hamiltonian twist map. In this case, the global minimizers have a special property: they can only cross once. As a nontrivial consequence, every one of them has the Birkhoff property. In crystals with a larger range of interaction and for higher order recurrence relations, the single crossing property does not hold and there can exist global minimizers that are not Birkhoff.

In this paper we investigate the crossings of global minimizers. Under a strong twist condition, we prove the following dichotomy: they are either Birkhoff, and thus very regular, or extremely irregular and nonphysical: they then grow exponentially and oscillate. For Birkhoff minimizers, we also prove certain strong ordering properties that are well known for twist maps.}\\

\section{Introduction}

The physical model that we take as the main motivation for the results of this paper, is a generalized Frenkel-Kontorova crystal model. The classical Frenkel-Kontorova model, first introduced in \cite{Frenkel-Kontorova}, can be used to describe an infinite array of particles that lie in a periodic background potential, where each particle is attracted to its closest neighbors by linear forces. Let a sequence $x = (..., x_{-1}, x_0, x_1,...)$ of real numbers describe the positions of the crystal particles, such that the position of the $i$-th particle is $x_i$. The equation of motion for this particle is given by $$m\frac{d^2x_i}{dt^2}=x_{i-1}-2x_i+x_{i+1}-V'(x_i),$$ where $V:\R\to \R$ satisfying $V(\xi+1)=V(\xi)$, is the periodic background potential.

To investigate the equilibrium solutions of this model, we have to solve for all $i\in \Z$ the recurrence relation \begin{equation}\label{f-k}x_{i-1}-2x_i+x_{i+1}-V'(x_i)=0.\end{equation} In \cite{AubryLeDaeron}, Aubry and le Daeron studied a particular set of equilibrium solutions of this model, the so-called global minimizers, or ground states. Global minimizers are, in a sense, quite a natural choice of solutions, since they ``minimize" the formal energy function of the crystal. Aubry and le Daeron proved that there exist global minimizers with any prescribed average spacing between particles. These solutions satisfy the ``Birkhoff'' property and are uniformly close to linear sequences. 

A surprising result in \cite{AubryLeDaeron} is that in fact all global minimizers of (\ref{f-k}) are Birkhoff, and hence very regular. This is a consequence of Aubry's Lemma, or the single crossing principle, which states that any two global minimizers of the Frenkel-Kontorova model can cross only once. More precisely, for a global minimizer $x\in \R^\Z$, let us picture the piecewise linear graph connecting the points $(i,x_i)\in \R^2$ by line segments. This is called Aubry's graph of $x$. The statement of Aubry's lemma is that Aubry's graphs of two global minimizers can cross in at most one point. At roughly the same time, similar results were obtained by Mather (\cite{Mather82}) using quite a different mathematical approach, and in the quite different setting of Hamiltonian twist maps. The correspondence is explained in \cite[4.2]{AubryLeDaeron}. 

It is possible to generalize the existence result of Birkhoff global minimizers of any rotation number to more complicated models than (\ref{f-k}). One generalization is to the case where the crystal is more-dimensional. It has been shown by Blank in \cite{blank} that for higher dimensional crystal models with nearest neighbor interactions, Birkhoff global minimizers of any rotation vector exist. The case where a particle also interacts with particles that are not its nearest neighbors, was addressed first in \cite{llave-lattices}. An analogous theory for elliptic PDEs on a torus was developed by Moser in \cite{moser86} and for geodesics on a 2-torus by Bangert in \cite{bangert90}. However, as first observed by Blank in \cite{blank} and \cite{blank2}, in most of these cases there are also global minimizers that are not Birkhoff. 

In this paper, we restrict ourselves to one-dimensional crystal models, where a particle interacts via elastic forces also with particles that are not its nearest neighbors. Such models were first considered in \cite{angenent88}. We call such a setting a generalized Frenkel-Kontorova model, or a finite-range variational monotone recurrence relation. In this setting, it is clear that Birkhoff global minimizers of all rotation numbers exist (see for example \cite{llave-lattices}). However, the main difference between a generalized Frenkel-Kontorova model and the classical Frenkel-Kontorova model is that the single crossing property does not hold anymore in the more general setting. In particular, there is no result stating that all global minimizers are Birkhoff. In fact, as we will show in section \ref{linear}, already in the setting of a linear generalized Frenkel-Kontorova model without a background potential, non-Birkhoff global minimizers exist. Because of this, we find the question of classifying global minimizers for generalized Frenkel-Kontorova model of interest. 

We moreover restrict ourselves to ``Newtonian crystal models'', for which Newton's second law applies. I.e., we assume that the forces acting on a particle can be represented as a sum of elastic forces arising from attraction to close-by particles. We will show with a dichotomy theorem that non-Birkhoff global minimizers have to be ``wild'' and so relatively non-physical. In particular, Birkhoff minimizers cannot be approximated by non-Birkhoff minimizers. This implies that when one is looking for properties of ``natural'' global minimizers of the generalized Frenkel-Kontorova models, it makes sense to study only the set of Birkhoff global minimizers more precisely. 

In addition, we want to investigate ordering properties for Birkhoff global minimizers of the generalized Frenkel-Kontorova model. In the case of the classical Frenkel-Kontorova model, a lot is known about ordering properties of global minimizers. As mentioned above, it turns out that Aubry's Lemma in the setting of twist maps implies that all global minimizers are Birkhoff, in other words, ordered with respect to their translates. In fact, more is true. It holds that all global minimizers of a fixed irrational rotation number are ordered and a slightly weaker statement holds also for rational rotation numbers. This was first shown by Aubry in \cite{AubryLeDaeron} and a nice overview of these results can be found in \cite{MatherForni}. We prove equivalent results for Birkhoff global minimizers of the generalized Frenkel-Kontorova model in the appendix to this paper.

\subsection{Discussion: minimal foliations and laminations}

A theorem by Bangert in \cite{bangert87} applied to generalized Frenkel-Kontorova models shows the set of Birkhoff minimizers of a specific irrational rotation number is strictly ordered, and is either connected (a minimal foliation), or it is disconnected (a minimal lamination). For irrational rotation numbers laminations form Cantor sets and are usually referred to as Cantori. 

The question of when a foliation and when a lamination can be expected, has been studied extensively. A reason in the case of classical Frenkel-Kontorova model is that minimal foliations correspond to energy-transport barriers of the corresponding Hamiltonian twist map - the standard map. The case where the class of global minimizers forms a foliation arises for example in the classical Frenkel-Kontorova model when the background potential is absent. There, in fact, the class of global minimizers of any rotation number forms a foliation. Moreover, if the rotation number of an invariant circle is ``very irrational", the KAM-theory provides perturbation results that show that for small enough smooth perturbations, the foliations persist (see \cite{SalamonZehnder}). A review of these results can be found in \cite{MatherForni}.

On the other hand, the case of Cantor sets for the classical Frenkel-Kontorova model arises in numerous examples. For example, for any irrational rotation number, the construction of the set of global minimizers as a continuation from the anti-integrable limit gives a Cantor set - see \cite{MackayMeiss}. In the setting of the standard map, the conditions that force the class of global minimizers of any irrational rotation number from a fixed interval to be a Cantor set, have been precisely studied in \cite{percival}. In the case where the rotation number is Liouville (not ``very irrational"), Mather has proved a much stronger result. It states that the set of local energies that have Cantor sets is dense in the $C^k$ topology for any $k \in \N$ - see \cite{mather_criterion}, \cite{mather_modulus} and \cite{mather_destruction}. Moreover, the equivalent results in the analytic case are worked out in \cite{Forni}.

For generalized Frenkel-Kontorova crystal models, the study of minimal foliations and laminations corresponds to the physical effects referred to sliding and pinning, respectively. The gaps in foliations define regions where atoms of the crystal that constitute a Birkhoff minimal solution cannot be found. Also in this general case, laminations can be obtained by large ``bumps'' on the local potentials (see for example \cite{ghost_circles}). Moreover, Mather's destruction result for Liouville rotation numbers  \cite{mather_destruction} has been generalized to this case by the authors in \cite{destruction}. 

However, since the single crossing property does not hold in this general setting, there are global minimizers that are not Birkhoff. The dichotomy theorem in this paper implies that at least in the setting we are working in, it makes sense to study minimal laminations and foliations, because Birkhoff global minimizers cannot be approximated by non-Birkhoff global minimizers. 

\subsection{Observations for a linear crystal model}\label{linear}

The first obvious extension of the Frenkel-Kontorova crystal model from (\ref{f-k}), is to assume that the atoms also interact with the second-closest neighbors via linear attracting forces. In this case the recurrence relation becomes \begin{equation}\label{f-k1}(1-b)x_{i-2}+bx_{i-1}-2x_i+bx_{i+1}+(1-b)x_{i+2}-V'(x_i)=0,\end{equation}  for some constant $b\in[0,1]$ and (\ref{f-k}) corresponds to the case where $b=1$. We set $V(\xi)\equiv 0$. Then it is easy to see by a convexity argument that any solution of (\ref{f-k1}) is a minimizer. Observe that all the solutions of (\ref{f-k}) can be described as linear sequences defined by $x_i:= \nu \cdot i+x_0$ and it is easy to see that linear sequences also solve \begin{equation}\label{rr1}(1-b)x_{i-2}+bx_{i-1}-2x_i+bx_{i+1}+(1-b)x_{i+2}=0\end{equation} for any $b\in[0,1)$. 

However, there are other solutions that we find by computing the general solutions of (\ref{rr1}), with the ansatz $x_i=c^i$ for some $c \in \mathbb{C}$. The equation we have to solve becomes $$(1-b)(c+c^{-1})^2+b(c+c^{-1})-4+2b=((1-b)(c+c^{-1})+2-b)(c+c^{-1}-2)=0.$$ This leads to the equations: $c+c^{-1}=2$ and $c+c^{-1}=-\frac{2-b}{1-b}$. The first equation has a double root in $c=1$, so it gives us the linear solutions. The second equation, in case $b\in (0,1)$, is solved by $$c_{0,1}=\frac{b-2\pm \sqrt{b(4-3b)}}{2(1-b)}$$ where $c_1=c_0^{-1}$. It follows that $c_{0}\in \R$, $c_0<0$ and $c_0^{-1}<0$. Then any solution $x$ of (\ref{rr1}) can be written as $x_i=k_0+k_1i+k_2c_0^i+k_3c_0^{-i}$. This implies that any global minimizer of (\ref{rr1}), where $b\in (0,1)$, is either linear, and in particular very regular, or exponentially growing and oscillating, and as such relatively non-physical. We will prove equivalent statements that reflect this duality in a much more general nonlinear setting. 

In case $b=0$, the equation $c+c^{-1}=-\frac{2-b}{1-b}$ has a double root in $c=-1$, so it gives the general solution $x$ by $x_i=k_0+k_1i+k_2(-1)^i+k_3(-1)^ii.$ Obviously, non-linear global minimizers in this case do not exhibit exponential growth. We will make assumptions on our model that exclude this degenerate uncoupled case.

\subsection{Setting}\label{setting}

In this section we introduce our notation and quote some standard results from Aubry-Mather Theory. 

As mentioned in the introduction, we are interested in monotone recurrence relations for which we assume that the particles obey Newton's second law of motion. More precisely, the force acting on a particular particle $x_i$ comprises of a local force arising from a background potential $V(x_i)$ and an interaction force that can be written as a sum of forces $\sum_jF_{i,j}$, such that $F_{i,j}$ corresponds to an elastic force generated by a nearby particles $x_j$. Moreover, we assume that the elastic forces are generated by potentials, which allows for a variational approach. This induces the following formal setup.

The underlying space for the variational principle is the space of real-valued sequences. Let $1\leq r\in \N$ be a natural number that represents the range of interaction between particles. Consider a $C^2$ function $S:\R^{r+1}\to \R$. For every sequence $x\in \R^\Z$ and for every $j\in \Z$ define the function $S_j(x):=S(x_{j},...,x_{j+r})$. We look for sequences $x$ that solve the following recurrence relations: \begin{equation}\label{rr}\sum_{j=i- r}^{i}\p_i S_j(x)=0, \ \ \forall i \in \Z.\end{equation} This is equivalent to finding solutions to the variational problem on the formal sum $$W(x)=\sum_{i\in \Z} S_i(x),$$ or solving the variational recurrence relation \begin{equation}\label{var_principle}\nabla W(x)=(\p_i W(x))_{i\in\Z}=\left(\sum_{j=i- r}^{i}\p_i S_j(x)\right)_{i\in\Z}\equiv 0.\end{equation} The formal potential $W$ corresponds to Newtonian variational monotone recurrence relations, if $S$ satisfies the definition of a ``local energy'', stated below.

\begin{definition}\label{S}
Let $1\leq r\in \N$ represent the range of interaction. We call a function $S\in C^2(R^{r+1})$ a local energy, if for $1\leq j \leq r$ there exist functions $f_j \in C^2(\R^2)$ such that $$S(\xi_1,...,\xi_{r+1})=\sum_{j=1}^r f_j(\xi_1,\xi_j)$$ and such that for every $1\leq j\leq r$, $f_j$ satisfies;
\begin{enumerate}
	 \item periodicity: $f_j(\nu+1,\mu+1)=f_j(\nu,\mu)$,
	\item uniform bound on the second derivatives: for all $i,k\in \{1,2\}$, there exists a constant $K>0$ such that $\|\p_{i,k}f_j\|_{\sup} \leq \frac{K}{r}$,
	\item coercivity: $f_j(\nu,\mu)\to \infty \text{ if } |\nu-\mu| \to \infty,$
	\item strong twist (monotonicity): there exists a $\lambda>0$ such that
	\begin{equation}\label{strong twist1} \p_1\p_2f_j(\nu,\mu) \leq -\lambda<0, \text{ for all } \nu,\mu \in \R. \end{equation}
\end{enumerate}
\end{definition}

\begin{remark}\label{remarkS}
Note that the conditions 1-4 in the definition \ref{S} imply that the local energies $S_i$ satisfy the following conditions;
\begin{enumerate}
	 \item periodicity: $S_i(x_{i}+1,...,x_{i+r}+1)=S_i(x_i,...,x_{i+r})$,
	\item uniform bound on the second derivatives: $\max\{j,k\in \Z \ | \  \|\p_{j,k}S_i\|_{\sup}\} \leq K$,
	\item coercivity: $S_i(x_{i},...,x_{i+r})\to \infty \text{ if } \sup_{i\leq j\leq i+r} |x_i-x_j| \to \infty,$
	\item strong twist (monotonicity): 
	\begin{equation}\begin{aligned}\label{strong twist} \p_i\p_j S_i(x)\leq-\lambda<0,& \text{ for all } j\in \{i+1,...,i+r\}, \ \text{ and }\\  \ \p_j\p_k S_i(x)\equiv 0, &\text{ if } j\neq i \text{ and }  k \neq i \text{ and } j\neq k.\end{aligned}\end{equation}
\end{enumerate}
\end{remark}

\begin{remark} To motivate these conditions, we explain what form the local energy for Frenkel-Kontorova models takes. By defining \begin{equation}\label{range1}S_i(x):=\frac{1}{2}(x_i-x_{i+1})^2 + V(x_i),\end{equation} where $V:\R \to \R$ is a real periodic $C^2$ function, (\ref{rr}) corresponds to (\ref{f-k}). Obviously, $S$ above satisfies all the conditions from Definition \ref{S}. The local energy corresponding to (\ref{f-k1}) is defined by \begin{equation}\label{range2}S_i(x):=\frac{b}{2}(x_i-x_{i+1})^2+\frac{1-b}{2}(x_i-x_{i+2})^2 + V(x_i).\end{equation} and again satisfies all of the conditions from Definition \ref{S}. Generalizing this model to the case where the forces are allowed to have non-linear dependence on the distance and to the case where the range of forces is arbitrary but finite, gives a general local energy from Definition \ref{S}. 
\end{remark}

Let us set some more notation. By $B=[i_0- r,i_1]$ we will denote an arbitrary finite segment of $\Z$ with $i_1-i_0\geq 0$. Next, denote by $\mathring{B}=[i_0,i_1]$ the interior of $B$ and by $\bar B:=[i_0- r,i_1+ r]$ its closure. Then we can define the boundary of $B$ by $\p B=\bar B\backslash \mathring{B}$ so that $\p B:=\p B_- \cup \p B_+$ and $\p B_-:=[i_0- r,i_0-1]$, $\p B_+:=[i_1+1,i_1+ r]$. 

We define $$W_B(x):=\sum_{i \in B}S_i(x)$$ which is a function of coordinates of $x$ with indices in $\bar B$, i.e. $x_{i_0- r},...,x_{i_1+ r}$. Observe that for any $i \in \mathring{B}$ it holds that $\p_i W_B(x)=\sum_{j=i-r}^{i}\p_i S_j(x)$. Hence, $x$ is a solution of (\ref{rr}), if and only if it is an equilibrium point for $W_B$, with respect to variations with support in $\mathring{B}$, for an arbitrary domain $B \subset \Z$. 

A strong condition that ensures that a sequence solves (\ref{rr}), is the following. 

\begin{definition}\label{glob_minimizer}A sequence $x$ is called a \textit{global minimizer}, if for all $B$ as above and all $v$ such that $supp(v) \subset \mathring{B}$ it holds that $W_B(x) \leq W_B(x+v).$ We denote the set of all global minimizers by $\M$. \end{definition}

Definition \ref{glob_minimizer} implies that global minimizers minimize an energy function with respect to compactly supported variations. In this sense, they are quite natural solutions for the problem (\ref{rr}). They are also the only solutions we are interested in for this paper.

The following definitions also prove useful. First, for every $k,l\in \Z$, define the translation operator \begin{equation}\label{translations}\tau_{k,l}:\R^\Z\to \R^\Z \ \text{ by } \ (\tau_{k,l}x)_i:=x_{i-k}+l.\end{equation} Moreover, we use the following notation for ordered sequences $x$ and $y$; 
\begin{itemize} 
	\item $x\leq y$: if for all $i\in \Z$, $x_i\leq y_i$,
	\item $x<y$: if for all $i\in \Z$, $x_i\leq y_i$ and $x\neq y$, (weak ordering)
	\item $x\ll y$: if for all $i\in \Z$, $x_i<y_i$ (strong ordering).
\end{itemize} 

Most of this paper is concerned with crossings of global minimizers. Let us make this more precise. Recall that we say that two sequences cross, if their Aubry graphs cross. To specify the domain in which crossings of sequences occur, we introduce the following definition.

\begin{definition}\label{domain of crossing} For sequences $x,y$, we call $D\subset \Z$ the domain of crossing of $x$ and $y$, if $D$ is an interval in $\Z$, i.e. $D=\varnothing$, $D=[j_0,j_1]$, $D=[j_0,\infty)$, $D=(-\infty,j_1]$ or $D=\Z$, and if the following holds. $D$ is the minimal interval such that $x<y$ or $y<x$ on $(-\infty,j_0]$ and that $x<y$ or $y<x$ on $[j_1,\infty)$. 

In other words, $x$ and $y$ are weakly ordered on all (at most both) ``connected'' components of the complement of $D$, but the ordering does not have to be the same on these components. \end{definition}

\subsection{Existence of global minimizers}\label{standard results}

In this section we give a brief sketch of how global minimizers are constructed when the local energy $S$ satisfies Definition \ref{S}. For more precise proofs we refer to \cite{llave-valdinoci}, or \cite{ghost_circles}.

The definition of translation in (\ref{translations}) allows us to define, for fixed integers $p,q \in \Z$, the set of $p$-$q$-periodic sequences by $$\X_{p,q}:=\{x\in \R^\Z \ | \ \tau_{p,q}x=x\}.$$ Since $\X_{p,q}$ is isomorphic to $\R^p$ and $S$ satisfies the periodicity condition from Definition \ref{S}, the formal action $W$ in the variational principle (\ref{var_principle}) can be replaced by the periodic action $W_{p,q}:=\sum_{i=1}^{p}S_i$ on $\X_{p,q}$. It is not difficult to show that the coercivity condition from Definition \ref{S} implies the existence of $p$-$q$-periodic sequences that minimize $W_{p,q}$. These sequences are called $p$-$q$-minimizers and they are solutions of (\ref{rr}). We denote the set of $p$-$q$-minimizers by $\M_{p,q}$. 

It turns out that periodic minimizers satisfy the following strong ordering properties. It follows by Aubry's lemma, applied in the setting of periodic sequences, that because of the twist condition (\ref{strong twist}), $p$-$q$-minimizers $x\neq y$ have to satisfy $x\ll y$ or $y\ll x$ (see for example Lemma 4.5 in \cite{ghost_circles}). Observe that for any $k,l \in \Z$, $\X_{p,q}$ is $\tau_{k,l}$ invariant and that also $W_{p,q}$ is $\tau_{k,l}$ invariant. In particular, it holds for every $x\in \M_{p,q}$ and every $k,l \in \Z$ that $\tau_{k,l}x\gg x$ or $\tau_{k,l}x\ll x$. This is the reason why periodic minimizers satisfy the well known \textit{Birkhoff property:} \begin{equation}\label{Birkhoff_def}\tau_{k,l}x \leq x \ \text{ or } \ \tau_{k,l}x \geq x \hspace{0.5cm} \text{holds for all } (k,l) \in \Z\times\Z.\end{equation} Every sequence $x$ with the Birkhoff property is called a \textit{Birkhoff sequence} and we denote the set of all Birkhoff sequences by $\B$. 

Furthermore, we denote the $p$-$q$-periodic Birkhoff sequences by $\B_{p,q}:=\B\cap \X_{p,q}$ and the set of Birkhoff global minimizers by $\BM:=\M\cap \B$. It can be shown that, because $p$-$q$-periodic minimizers are Birkhoff, they are also global minimizers, so that $\M_{p,q}\subset \M \cap \X_{p,q}$. In fact, also the inclusion in the other direction holds, so that $\M_{p,q}=\M \cap \X_{p,q}$. Proofs of the statements above can be found in $\S 4$ \cite{ghost_circles}.

Next, we recall some properties of Birkhoff sequences in general. It is well known that Birkhoff sequences have a \textit{rotation number} $$\rho(x):=\lim_{n\to \pm \infty}\frac{x_n}{n}$$ and that they satisfy the uniform estimate \begin{equation}\label{rotation number estimate}|x_n-x_0-\rho(x)n|\leq 1 \ \text{ for all } n\in \Z \end{equation} (see $\S 9$ \cite{gole01}). Denote $\B_\nu:=\{x \in \B \ | \ \rho(x)=\nu\}$ and $\BM_\nu:=\B_\nu \cap \M$ and observe that for any $x\in \B_{p,q}$, $\rho(x)=\frac{q}{p}$. As discussed above, $p$-$q$-periodic Birkhoff minimizers of every period exist, so $\BM_{q/p}\neq \varnothing$. The uniform estimate (\ref{rotation number estimate}) and the Birkhoff property (\ref{Birkhoff_def}), together with definition \ref{glob_minimizer}, shows that $\BM$ is compact with respect to point-wise convergence. This implies that we can take limits of periodic minimizers and get global minimizers of any irrational rotation number. We state this result, first published in \cite{AubryLeDaeron}, in the following theorem.

\begin{theorem}[Existence of Birkhoff global minimizers]
For any local energy $S$ that satisfies Definition \ref{S} and any rotation number $\nu\in \R$, there are Birkhoff global minimizers with rotation number $\nu$, i.e. $\BM_\nu \neq \varnothing$. 
\end{theorem}

\subsection{Outline of the paper and statement of the results} \label{outline}

In Section \ref{section preliminaries} we assemble all the tools needed for the proofs of Theorem A and Theorem B, after giving an intuitive explanation of the ideas behind these proofs. Section \ref{section bounded} contains the proof of Theorem A, stated below. Recall the definition \ref{domain of crossing} of the domain of crossing for sequences $x$ and $y$. 
\begin{theorema} Let $x,y\in \M$ and assume that for the domain of crossing $D$ of $x$ and $y$ the following holds: $D\neq \varnothing$ and $|D|<\infty$. Then $|D|<\tilde K$, where the constant $\tilde K$ depends only on the range of interaction $r$ and the uniform constants $\lambda$ and $K$ from Definition \ref{S}. \end{theorema}

In other words, we show that if the domain of crossing for two global minimizers $x$ and $y$ is bounded, then its size is smaller than some uniform constant $\tilde K$, independent of $x$ and $y$.

In Section \ref{section unbounded domain} we push the idea of the proof of Theorem A, to get the following result.
\begin{theoremb}
Assume that the domain of crossing $D$ for $x,y\in \M$ is infinite. Then there is a constant $d \in \N$ that depends only on the range of interaction $r$ and the uniform constants $\lambda$ and $K$ from Definition \ref{S}, such that the following holds. There exist monotone sequences $k_n,l_n \in D$ with $|k_{n+1}-k_n|\leq d$ and $|l_n-k_n|\leq r$ so that $$\text{ for all } n, \  x_{k_n}>y_{k_n}, \ x_{l_n}<y_{l_n}\ \text{ and }  \ (x_{k_n}-y_{k_n})(y_{l_n}-x_{l_n})\geq 2^n.$$
\end{theoremb}

This theorem is the counterpart of Theorem A. It says that if the domain of intersection for global minimizers $x$ and $y$ is infinite, then $x-y$ behaves very wildly in some specific sense. In fact, a monotone subsequence of the sequence $x-y$ grows exponentially and changes sign.

In Section \ref{section birkhoff} we compare global minimizers to their translates and apply Theorem A and Theorem B. This results in the following:
 
\begin{theoremd}
For every global minimizer $x\in \M$ one of the following two cases is true: 

\begin{itemize}
	\item It holds that $x$ is a Birkhoff global minimizer and thus very regular.
	\item It holds that $x$ is not a Birkhoff global minimizer. Then $x$ is very irregular in the following sense. There are monotone infinite sequences $\{k_n, l_n\}\in \Z$, with $|k_{n+1}-k_{n}|\leq d$, $|l_n-k_n|\leq r$ such that one of the following inequalities holds for all $n\in \N$: \begin{align*}
	(x_{k_n+1}-x_{k_n} + 1)(x_{l_n}-x_{l_n+1}+ 1)&\geq 2^n, \text{ or } \\ (x_{k_n+1}-x_{k_n} - 1)(x_{l_n}-x_{l_n+1}- 1)&\geq 2^n.\end{align*} In particular, for every $n$ one of the following must hold: $$x_{k_n+1}-x_{k_n}\geq 2^{n/2}-1, \ \text{ or } \ x_{l_n}-x_{l_n+1}\geq 2^{n/2}-1.$$
\end{itemize}
\end{theoremd}

A global minimizer is thus either very regular and ``almost linear'', or it is oscillating and exponentially growing. 

\vspace{0.3cm}
\noindent
\textbf{Appendix:}\\
 For global minimizers of twist maps, it is not only known that they are Birkhoff, but also that they exhibit some stronger ordering properties (see \cite{MatherForni}). We develop the equivalent theory for our setting in Appendix \ref{section ordered minimizers}. We compare arbitrary Birkhoff global minimizers of the same rotation number. We work in the space of Birkhoff global minimizers $\BM$ and assume that a weaker twist condition holds, making the statements slightly more general. We write the collection of Birkhoff global minimizers as the following union $$\BM:=\bigcup_{\nu \in\R\backslash \Q  }\BM_\nu\cup \bigcup_{q/p \in \Q} \BM_{q/p}^+ \cup \BM_{q/p}^-,$$ defined by: 
\begin{itemize} 
	\item for $\nu\in \R\backslash \Q$, $\BM_\nu:=\{x\in \M\cap \B_\nu\}$,
	\item for $p,q\in \Z$, $\BM_{q/p}^+:=\{x\in \M\cap \B_{q/p}\ | \ \tau_{p,q}x\geq x\}$ and
	\item for $p,q\in \Z$, $\BM_{q/p}^-:=\{x\in \M\cap \B_{q/p}\ | \ \tau_{p,q}x\leq x\}$.
\end{itemize}
Using the ideas from the classical Aubry-Mather Theory for twist maps, we will show that each of the sets $\BM_\nu$, $\BM_{q/p}^+$ and $\BM_{q/p}^-$ is ordered. Moreover, we show that whenever there is a gap $[x^-,x^+]$ in $\M_{p,q}=\BM^+_{q/p}\cap \BM^-_{q/p}$, then it contains heteroclinic connections in $\BM^+_{q/p}\backslash \M_{p,q}$ and in $\BM^-_{q/p}\backslash \M_{p,q}$, connecting $x^-$ and $x^+$.

\section{Preliminaries}\label{section preliminaries}

\subsection{Minimum-maximum principle}\label{section min-max}

In this section, we explain some basic results that are the main tools for the rest of this paper. In particular, we derive the so-called minimum-maximum principle, strong comparison principle and an analogue of Aubry's lemma (Lemma \ref{even_crossings}), for the local energy $S$ as in Definition \ref{S}. We start with the following definition.

\begin{definition}\label{crossing energy}
For $x,y\in \R^\Z$, define $M$ and $m$ by $M_i:=\max\{x_i,y_i\}$ and $m_i:=\min\{x_i,y_i\}$. 

We call $W_B^c(x,y):=W_B(y)-W_B(m)-W_B(M)+W_B(x)$ \textit{the crossing energy of $x$ and $y$ on $B$.}
\end{definition}

To compute the crossing energy of $x$ and $y$, we use the idea from \cite{llave-valdinoci}, that allows us to generalize the so-called minimum-maximum principle from classical Aubry-Mather Theory to our setting. Define \begin{equation}\label{alpha-beta}\alpha_i:=\left\{ \begin{array}{lll} y_i-x_i & \mbox{if} & y_i-x_i>0,  \\ 0 & \mbox{else;}&\end{array} \right. \hspace{0.5cm} \beta_i:=\left\{ \begin{array}{lll} y_i-x_i & \mbox{if} & y_i-x_i<0,  \\ 0 & \mbox{else.}&\end{array} \right.\end{equation} Then it holds that $M=\max\{x,y\}=x+\alpha$, $m=\min\{x,y\}=x+\beta$ and $y=x+\alpha+\beta$. This allows us to prove the following.

\begin{lemma}[Minimum-maximum principle]\label{min-max lemma}For an arbitrary finite segment $B \subset \Z$ it holds that $W_B^c(x,y)\geq0$, i.e. $W_B(x)+W_B(y)\geq W_B(M)+W_B(m)$. \end{lemma}

\begin{proof} By interpolating $W_B^c(x,y)$ with respect to $\alpha$ and $\beta$, we get \begin{equation*}\begin{aligned}W_B^c(x,y)&=W_B(y)-W_B(m)-W_B(M)+W_B(x)=\\&= \sum_{i\in B}\int_0^1 \int_0^1\frac{d}{dt}\frac{d}{ds}S_i(x+t\alpha + s\beta)ds \: dt =\\&=\sum_{i\in B} \sum_{j,k=i}^{i+ r}\int_0^1 \int_0^1\p_{j,k}S_i(x+t\alpha + s\beta)ds \: dt\ \alpha_j\beta_k.\end{aligned}\end{equation*}
Note that in the sum above $\alpha_i\beta_j\leq0$ and that the supports of $\alpha$ and $\beta$ are disjoint, so all of the terms with non-mixed derivatives vanish. Moreover, it follows from the strong twist condition (\ref{strong twist}), that non-zero terms in the formula above arise only in the case when either $j=i$, or $k=i$. By the uniform bounds from Definition \ref{S}, this gives the following inequality: \begin{equation}\begin{aligned} W_B^c(x,y)=&\sum_{i\in B} \sum_{j=i}^{i+ r}\int_0^1 \int_0^1\p_{j,i}S_i(x+t\alpha + s\beta)ds \: dt\ (\alpha_i\beta_j+\alpha_j\beta_i) \\ \geq & -\lambda \sum_{i\in B} \sum_{j=i}^{i+ r} (\alpha_j\beta_i+\alpha_i\beta_j). \label{min-max} \end{aligned}\end{equation} In particular, since $\beta\leq 0$ and $\alpha\geq0$, this implies that $W_B^c(x,y)\geq0$, so $W_B(x)+W_B(y)\geq W_B(m)+W_B(M)$. 
\end{proof}

In fact, it is clear from the proof above that $W_B(x)+W_B(y)> W_B(m)+W_B(M)$, whenever such $i,j\in\Z$ exist that $|i-j|\leq r$ and $\alpha_i\beta_j<0$ or $\alpha_j\beta_i<0$. This inequality means that any crossing of the sequences $x,y$ is reflected in the value of $W_B^c(x,y)$. This is a consequence of the strong twist condition (\ref{strong twist}) and also the reason why a weaker twist condition, as in \cite{llave-lattices} or \cite{ghost_circles} cannot be used in the following proofs.

Next, we explain an important property of solutions of the variational principle (\ref{var_principle}).

\begin{lemma}[Strong ordering property]\label{strong ordering}
Let $B\subset \Z$ and let $x$ and $y$ be solutions of the recurrence relation (\ref{rr}) for all $i\in \mathring{B}$. Then it holds that if $x< y$ on $B$, then $x\ll y$ on $\mathring{B}$. 
\end{lemma}

\begin{proof}
Since $x< y$ on $B$, it follows that $y_i-x_i=\alpha_i$ for all $i\in B$. It must hold for every $i \in \mathring{B}$ that \begin{equation}\label{x-y} \begin{aligned}0=&\p_{i}W(y)-\p_{i}W(x)=\sum_{j=i- r}^{i}(\p_{i} S_{j}(y)-\p_i S_j(x))=\\ =&\sum_{j=i- r}^{i}\sum_{k=j}^{j+ r}\int_0^1\p_{k,i}S_j(\tau y+(1-\tau)x)d \tau \alpha_k =\\ = & \sum_{j=i- r}^{i}\int_0^1\p_{j,i}S_j(\tau x+(1-\tau)y)d \tau \alpha_j+\sum_{j=i}^{i+ r}\int_0^1\p_{j,i}S_i(\tau x+(1-\tau)y)d \tau \alpha_j.\end{aligned} \end{equation} The third equality follows from the strong twist condition (\ref{strong twist}), by setting $k=j$ for the first sum, and $j=i$ followed by $k=j$ for the second sum. 

Assume now that there is an $i\in \mathring{B}$ with $\alpha_i=0$. Then, by (\ref{strong twist}), all the second derivatives in (\ref{x-y}) are strictly negative and since $\alpha_j\geq 0$ for all $j$, it must follow that $\alpha_j=0$ for all $j\in [i- r, i+ r]$. By induction, it follows that $x=y$ on $B$, a contradiction because we assumed that $x<y$ on $B$, so it must hold that $\alpha_i>0$ for all $i\in \mathring{B}$.
\end{proof}

Applying Lemma \ref{strong ordering} gives the following corollary. 

\begin{corollary}\label{ordering minimizers}
Assume that $x\neq y$ are two solutions of (\ref{rr}) such that $x>y$. Then $x\gg y$.
\end{corollary}

The estimate (\ref{min-max}) from Lemma \ref{min-max lemma} and Corollary \ref{ordering minimizers} now give us the means to analyze more precisely, how two global minimizers cross in a specific domain. 

In the remainder of the text, the following notation will prove useful. 

\begin{definition}\label{min-max-variations}
Let $B\subset \Z$ be arbitrary, but fixed. Define 
\begin{equation}\begin{aligned}M^B_i(x):=\left\{ \begin{array}{lll} x_i & \mbox{if} & i \notin \mathring{B},  \\ M_i & \mbox{if}& i \in \mathring{B};\end{array} \right. \hspace{0.5cm} M^B_i(y):=\left\{ \begin{array}{lll} y_i & \mbox{if} & i \notin \mathring{B},  \\ M_i & \mbox{if}& i \in \mathring{B};\end{array} \right.\nonumber \\ m^B_i(x):=\left\{ \begin{array}{lll} x_i & \mbox{if} & i \notin \mathring{B},  \\ m_i & \mbox{if}& i \in \mathring{B};\end{array} \right. \hspace{0.5cm} m^B_i(y):=\left\{ \begin{array}{lll} y_i & \mbox{if} & i \notin \mathring{B},  \\ m_i & \mbox{if}& i \in \mathring{B}.\end{array} \right.\end{aligned}\end{equation} By this definition we changed $M$ and $m$ into variations of $x$ and $y$ with support in $\mathring{B}$. 
\end{definition}

\begin{lemma}\label{even_crossings} 
Let $i_0<k_0<k_1<i_1$ be integers such that $i_0\leq k_0- r$ and $i_1\geq k_1+ r$. If $x$ and $y$ are global minimizers, such that $x_i\leq y_i$ for all $i \in [i_0,k_0-1] \cup [k_1+1,i_1]$, then $x\ll y$ on $[k_0,k_1]$.
\end{lemma}

\begin{proof}
Let $B:=[k_0- r,k_1]$, so that $\mathring{B}=[k_0,k_1]$ and that $m^B(x)$ and $M^B(y)$ are variations of $x$ and $y$ respectively, with support in $\mathring{B}$. Observe that by assumption, $M^B(y)=M$ and $m^B(x)=m$ on $\p B=[k_0- r,k_0-1]\cup[k_1+1,k_1+ r]$ and so by definition also on the whole $\bar B$. Recall that $W_B(x)$ is a function that depends only on terms of $x$ that have indices in $\bar B$. So it must hold by Lemma \ref{min-max lemma} and by the definition of global minimizers (Definition \ref{glob_minimizer}) that $W_B(x)=W_B(m^B(x))$ and $W_B(y)=W_B(M^B(y))$. This implies that also $m^B(x)$ and $M^B(y)$ are global minimizers. Since it holds that $x\geq m^B(x)$, but not $x\gg m^B(x)$, Corollary \ref{ordering minimizers} implies that $x\equiv m^B(x)$. So, on $\bar B$ it holds that $x<y$ and by Lemma \ref{strong ordering}, it then holds that $x\ll y$ on $\mathring{B}$.
\end{proof}

\begin{corollary}[Aubry's Lemma]\label{Aubry's Lemma}
Assume that the local energy $S$ satisfies Definition \ref{S} with the range $r=1$ and assume that $x\neq y$ are global minimizers for $S$. Then $x$ and $y$ cross at most once, i.e. $D=i_0$ or $D=\varnothing$.
\end{corollary}

\begin{proof}
Lemma \ref{even_crossings} in this case implies that if there exist indices $i_0 \in \Z$ and $i_1 \in \Z$ such that $x_{i_0}\geq y_{i_0}$ and $x_{i_1}\geq y_{i_1}$, then $x> y$ on $[i_0,i_1]$. This easily implies the statement.
\end{proof}

Corollary \ref{Aubry's Lemma} shows that Lemma \ref{even_crossings} implies Aubry's lemma, or the single crossing principle in the case of twist maps. In case of $ r>1$, it has some more subtle consequences.

\vspace{0.5cm}
\noindent
\textbf{Implications of Lemma \ref{even_crossings}:} 

Recall definition \ref{domain of crossing} of the domain of crossing. Lemma \ref{even_crossings} immediately implies the following corollary, which we state without a proof.

\begin{corollary}\label{crosses}
Let $D$ be the domain of crossing for $x$ and $y$. If $D$ is bounded and $x>y$ on $\Z\backslash D$, then $D=\varnothing$.
\end{corollary}

Let $D=[j_0,j_1]\neq \varnothing$ be bounded. Then by Corollary \ref{crosses}, $x\geq y$ on $(-\infty,j_0]$ implies that $y\geq x$ on $[j_1,\infty)$. In particular, we may assume without loss of generality that if $D=[j_0,j_1]\neq \varnothing$ is bounded, then $x\leq y $ (or equivalently $\beta=0$) on $(-\infty,j_0-1]$ and $x\geq y$ (or equivalently $\alpha=0$) on $[j_1+1,\infty)$. I.e., we assume that $j_0:=\min\{i \in \Z \ | \ \beta_i<0\}$ and $j_1:=\max\{i \in \Z \ | \ \alpha_i>0\}$. This will be our assumption in Section \ref{section bounded}.

Moreover, in case the domain of crossing of $x$ and $y$, $D=[j_0,j_1]\neq \varnothing$ is bounded, applying Lemma \ref{even_crossings} with either $k_0=j_0$, or $k_1=j_1$ and reversing the roles of $x$ and $y$ if necessary, the definition of $j_0$ and $j_1$ gives us the following corollary.

\begin{corollary}\label{dense_crossings}
If $D=[j_0,j_1]\neq \varnothing$ is bounded, there is no segment $I \subset [j_0-r+1,j_1 + r-1]$ with $|I|=r$, such that $\alpha|_I \equiv 0$ or $\beta|_I \equiv 0$.
\end{corollary}

In case the domain of crossing $D$ of $x$ and $y$ is unbounded, the equivalent statement that follows from Lemma \ref{even_crossings} is the following.

\begin{proposition}\label{infinite_oscillations}
Let $D$ be the domain of crossing for $x$ and $y$. If $D$ is unbounded, then there exists an unbounded domain $\tilde D\subset D$, such that there is no segment $I \subset \tilde D$ with $|I|=r$, such that $\alpha|_I \equiv 0$ or $\beta|_I \equiv 0$.
\end{proposition}

\begin{proof}
Let $D$ be the domain of crossing for global minimizers $x$ and $y$, as in Definition \ref{domain of crossing}. By Lemma \ref{even_crossings} it holds that there is at most one segment $[i_l,i_r]=I\subset D$ with $i_r-i_l\geq r$, such that $\alpha|_I \equiv 0$. Similarly, there is at most one segment $J=[j_l,j_r]\subset D$ with $j_r-j_l \geq r$ such that $\beta|_J\equiv 0$, so we may take the unbounded domain $\tilde D$, such that it does not include any of those two segments. (Moreover, the proof of Theorem A will show that if there are such segments $I$ and $J$, then $|i_r-j_l| \leq \tilde{K}$, where $\tilde K$ is defined in Theorem A.)
\end{proof}

\subsection{The idea of the proofs}

Now we roughly explain the idea behind the proofs of Theorem A and Theorem B. 

Let $D$ be the domain of crossing for $x$ and $y$ and let $I\subset D$ be such that $|I|=r$, but otherwise arbitrary. By Corollary \ref{dense_crossings} it holds that there are indices $j,k\in I$ such that $\alpha_j > 0$ and that $\beta_k<0$. Equivalently, this holds for for every $I \in \tilde D$, where $\tilde D$ is as in Proposition \ref{infinite_oscillations}. Hence, if we assume that for some $i\in D$, $\beta_i<0$, then there exists an index $j\in [i,i+ r]$, such that $\alpha_j>0$ and similarly, if $\alpha_i>0$, there exists a $j\in [i,i+ r]$ such that $\beta_j<0$. This means that the sequences $x$ and $y$ cross between $i$ and $j$ and moreover, by (\ref{min-max}), that the crossing energy $W_B^c(x,y)$ is positive, as soon as $B \cap D \neq \varnothing$. This also implies that $W_B^c(x,y)$ grows proportionally to the size of $B \cap D \neq \varnothing$, where $\alpha_i\beta_j$ terms determine the growth rate.

Since $M^B(x)$ or $M^B(y)$ and $m^B(x)$ or $m^B(y)$ are variations of $x$ or $y$ with support in $\mathring{B}$ and because $x$ and $y$ are global minimizers, it must moreover hold for every $B$ that \begin{equation*}\begin{aligned}W_B(x)+W_B(y)\leq &W_B(M_B(x))+W_B(m_B(y)) \text{ and }\\ W_B(x)+W_B(y)\leq &W_B(M_B(y))+W_B(m_B(x)).\end{aligned}\end{equation*} Equivalently, (since $\max\{M^B(x),m^B(y)\}=M$, etc.) we can subtract $W_B(M)+W_B(m)$ on both sides of both inequalities, and write \begin{equation}\label{gp}W_B^c(x,y)\leq W^c_B(M_B(x), m_B(y)) \ \text{ and }  \ W_B^c(x,y)\leq W^c_B(M_B(y), m_B(x)).\end{equation}

Because of the following observation, we view (\ref{gp}) as the ``general principle'' of the proof. Recall that $W_B(z)$ depends only on $z_i$ with $i\in \bar B$. Moreover, it follows from Definition \ref{min-max-variations} that $M^B(y)\equiv M^B(x)\equiv M$ and $m^B(x)\equiv m^B(y) \equiv m$ on $\mathring{B}$. Then it must hold, by a similar inequality as (\ref{min-max}), that $W^c_B(M_B(y), m_B(x))$ and $W^c_B(M_B(x), m_B(y))$ depend on finitely many $\alpha$ and $\beta$ terms around $\p B$, i.e. a fixed number of terms of $x-y$ around $i_0$ and $i_1$. In view of this, we call $W^c_B(M_B(y), m_B(x))$ and $W^c_B(M_B(x), m_B(y))$ ``the boundary energies''. In fact, it turns out that the terms that arise in the boundary energies, can be estimated by a finite number of $\alpha_i\beta_j$ terms, for some indices $i,j$ close to $\p B$. These estimates are obtained in Section \ref{section boundary energy} and are the most technical part of this paper. 

These considerations together with (\ref{gp}) imply that for a large domain $B$, the products of a small number of $\alpha$ and $\beta$ terms around $\p B$ must have a value proportional to all the products of $\alpha$ and $\beta$ terms in (\ref{min-max}). Hence, this small number of terms must exhibit an exponential growth in the case that $D$ is unbounded and they give a uniform bound on the size of $D$, if $D$ is bounded.

\subsection{Estimates for the boundary energies}\label{section boundary energy}

The goal of this section is to estimate the boundary energies $W_B^c(M_B(x),m_B(y))$ and $W_B^c(M_B(y),m_B(x))$.

\begin{definition}\label{alpha-beta-var}
Define $\alpha^B(x):=M-M^B(x)$, $\beta^B(x):=m-M^B(x)$, $\alpha^B(y):=M-M^B(y)$ and $\beta^B(y):=m-M^B(y)$. 
\end{definition}

\begin{remark}\label{comparing} It follows directly from the definition of $M^B(x)$ etc. in Definition \ref{min-max-variations} and from the definition of $\alpha$ and $\beta$ (\ref{alpha-beta}) that $\alpha^B(x) \equiv 0$ on $\mathring{B}$ and $\alpha^B(x) \equiv \alpha$ else, and that $\beta^B(x) \equiv \beta - \alpha$ on $\mathring{B}$ and $\beta^B(x) \equiv \beta$ otherwise. Similarly, $\alpha^B(y) \equiv 0$ on $\mathring{B}$ and $\alpha^B(y) \equiv -\beta$ else, and that $\beta^B(y) \equiv \beta - \alpha$ on $\mathring{B}$ and $\beta^B(y) \equiv -\alpha$ otherwise. Moreover, notice that $m^B(y)=M^B(x)+\alpha^B(x)+\beta^B(x)$ and $m^B(x)=M^B(y)+\alpha^B(y)+\beta^B(y)$. \end{remark}

For the sake of brevity, let us denote \begin{equation*}\begin{aligned}I^{i,j}_B(x):=\int_0^1\int_0^1 \p_{i,j}S_i(M^B(x)+t\alpha^B(x)+s\beta^B(x))ds dt;\\I^{i,j}_B(y):=\int_0^1\int_0^1 \p_{i,j}S_i(M^B(y)+t\alpha^B(y)+s\beta^B(y))ds dt.\end{aligned}\end{equation*} Computing the crossing energy from definition \ref{crossing energy} gives us similarly as in (\ref{min-max})\begin{equation*}\begin{aligned}W_B^c(M^B(x),m^B(y))=\sum_{i\in B}\sum_{j=i}^{i+r}I^{i,j}_B(x)(\beta^B(x)_i\alpha^B(x)_j+\beta^B(x)_j\alpha^B(x)_i);\\W_B^c(M^B(y),m^B(x))=\sum_{i\in B}\sum_{j=i}^{i+r}I^{i,j}_B(y)(\beta^B(y)_i\alpha^B(y)_j+\beta^B(y)_j\alpha^B(y)_i).\end{aligned}\end{equation*} 

\begin{proposition}\label{boundary energy}
For every domain $B=[i_0-r,i_1]$ with $i_1-i_0>2 r$, the boundary energies can be split in the following way. $$W_B^c(M^B(x),m^B(y))=W_{i_0,-}^b+W_{i_1,+}^b \ \text{ and } \ W_B^c(M^B(y),m^B(x))=\tilde W_{i_0,-}^b+\tilde W_{i_1,+}^b,$$ where the energies $W_{i_0,-}^b$ and $\tilde W_{i_0,-}^b$ depend only on terms of $x$ and $y$ with indices ``close to'' $\p B_-$, and $W_{i_1,+}^b$ and $\tilde W_{i_1,+}^b$ depend only on terms of $x$ and $y$ with indices ``close to'' $\p B_+$. 

Furthermore, these energies can be split into ``mixed'' $\alpha_i\beta_j$ terms, and ``double'' $\alpha_i\alpha_j$ or $\beta_i\beta_j$ terms by $$W_{i_0,-}^b = S_{i_0,-}^{mix}+ S_{i_0,-}^{dbl} \ \text{ and } \ W_{i_1,+}^b = S_{i_1,+}^{mix}+ S_{i_1,+}^{dbl},$$ $$\tilde W_{i_0,-}^b = \tilde S_{i_0,-}^{mix}+ \tilde S_{i_0,-}^{dbl} \ \text{ and } \ \tilde W_{i_1,+}^b = \tilde S_{i_1,+}^{mix}+ \tilde S_{i_1,+}^{dbl}$$ given by
\begin{align*}
S_{i_0,-}^{mix}& :=\sum_{i=i_0-r}^{i_0-1}\sum_{j=i}^{i+r}I_B^{i,j}(x)\alpha_i\beta_j + \sum_{i=i_0-r}^{i_0-1}\sum_{j=i}^{i_0-1}I_B^{i,j}(x)\beta_i\alpha_j,\\
S_{i_0,-}^{dbl}& :=\sum_{i=i_0-r}^{i_0-1}\sum_{j=i_0}^{i+r} I_B^{i,j}(x)\alpha_i\alpha_j,\\
S_{i_1,+}^{mix}& :=\sum_{i=i_1+1}^{i_1}\sum_{j=i}^{i+r}I_B^{i,j}(x)\alpha_i\beta_j + \sum_{i=i_1-r+1}^{i_1}\sum_{j=i_1+1}^{i+r}I_B^{i,j}(x)\beta_i\alpha_j,\\
S_{i_1,+}^{dbl}& :=\sum_{i=i_1-r+1}^{i_1}\sum_{j=i_1+1}^{i+r}I_B^{i,j}(x)\alpha_i\alpha_j,\\
\tilde S_{i_0,-}^{mix}& :=\sum_{i=i_0-r}^{i_0-1}\sum_{j=i}^{i+r}I_B^{i,j}(y)\beta_i\alpha_j + \sum_{i=i_0-r}^{i_0-1}\sum_{j=i}^{i_0-1}I_B^{i,j}(y)\alpha_i\beta_j,\\
\tilde S_{i_0,-}^{dbl}& :=\sum_{i=i_0-r}^{i_0-1}\sum_{j=i_0}^{i+r}I_B^{i,j}(y)\beta_i\beta_j,\\
\tilde S_{i_1,+}^{mix}& :=\sum_{i=i_1+1}^{i_1}\sum_{j=i}^{i+r}I_B^{i,j}(y)\beta_i\alpha_j + \sum_{i=i_1-r+1}^{i_1}\sum_{j=i_1+1}^{i+r}I_B^{i,j}(y)\alpha_i\beta_j,\\
\tilde S_{i_1,+}^{dbl}& :=\sum_{i=i_1-r+1}^{i_1}\sum_{j=i_1+1}^{i+r}I_B^{i,j}(y)\beta_i\beta_j.
\end{align*}
\end{proposition}

\begin{proof}
We compute the representation of $W_{i_0,-}^b$. The crossing energy takes the form $$W_B^c(M^B(x),m^B(y))=\sum_{i=i_0-r}^{i_1}\sum_{j=i}^{i+r}I^{i,j}_B(x)(\alpha^B(x)_i\beta^B(x)_j+\alpha^B(x)_j\beta^B(x)_i).$$ Since $\alpha^B(x)|_{\mathring{B}}\equiv 0$ and $i_1-i_0> 2r$ it is clear that we can split the crossing energy into $$W_B^c(M^B(x),m^B(y))=W_{i_0,-}^b+W_{i_1,+}^b.$$ More precisely, because $\alpha^B(x)_i=0$ for all $i\geq i_0$, we can split the terms in $W_{i_0,-}^b$ in the following way: \begin{align*}W_{i_0,-}^b=&\sum_{i=i_0-r}^{i_0-1}\sum_{j=i}^{i+r}I^{i,j}_B(x)\alpha^B(x)_i\beta^B(x)_j+ \sum_{i=i_0-r}^{i_0-1}\sum_{j=i}^{i_0-1}I^{i,j}_B(x)\beta^B(x)_i\alpha^B(x)_j=\\
=&\sum_{i=i_0-r}^{i_0-1}\sum_{j=i}^{i_0-1}I_B^{i,j}(x)\alpha_i\beta_j+\sum_{i=i_0-r}^{i_0-1}\sum_{j=i_0}^{i+r}I_B^{i,j}(x)\alpha_i(\beta_j-\alpha_j)+\sum_{i=i_0-r}^{i_0-1}\sum_{j=i}^{i_0-1}I_B^{i,j}(x)\beta_i \alpha_j=\\
=&\sum_{i=i_0-r}^{i_0-1}\sum_{j=i}^{i+r}I_B^{i,j}(x)\alpha_i\beta_j+\sum_{i=i_0-r}^{i_0-1}\sum_{j=i}^{i_0-1}\beta_i\alpha_j+ \sum_{i=i_0-r}^{i_0-1}\sum_{j=i_0}^{i+r}\alpha_i\alpha_j
\end{align*}
The calculations above follow from Remark \ref{comparing}. Similar considerations gives the other equalities in the proposition.
\end{proof}
 
To make use of the general principle of the proof (\ref{gp}), we need to compare $W_B^c(x,y)$ and $W_B^c(M^B(x),m^B(y))$. Hence, we need to be able to compare all the terms from Proposition \ref{boundary energy} to terms from $W_B^c(x,y)$. 

First of all, we use the uniform estimate on the second derivatives from definition \ref{S}, to get $I^{i,j}_B(y) \leq K$ and $I^{i,j}_B(x)\leq K$. Next, define \begin{equation}\label{def-mix} E_{i_0,-}^{mix}:=\sum_{i=i_0-r}^{i_0-1}\sum_{j=i}^{i+r}\alpha_i\beta_j + \sum_{i=i_0-r}^{i_0-1}\sum_{j=i}^{i_0-1}\beta_i\alpha_j, \end{equation} where the sums correspond to the sums from $S_{i_0,-}^{mix}$. In the analogous way we define also $E_{i_1,+}^{mix}$, $\tilde E_{i_0,-}^{mix}$ and $\tilde E_{i_1,+}^{mix}$, corresponding to $S_{i_0,-}^{mix}$, $\tilde S_{i_0,-}^{mix}$ and $\tilde S_{i_1,+}^{mix}$. Then it holds by the uniform estimates from definition \ref{S}, because the supports of $\alpha$ and $\beta$ are disjoint, that \begin{equation} \label{mixed estimate}\begin{aligned}\lambda E_{i_0,-}^{mix}\leq S_{i_0,-}^{mix}\leq K E_{i_0,-}^{mix} \ \text{ and } \ \lambda E_{i_1,+}^{mix}\leq S_{i_1,+}^{mix}\leq K E_{i_1,+}^{mix}, \\ \lambda \tilde E_{i_0,-}^{mix}\leq \tilde S_{i_0,-}^{mix}\leq K \tilde E_{i_0,-}^{mix} \ \text{ and } \ \lambda \tilde E_{i_1,+}^{mix}\leq \tilde S_{i_1,+}^{mix}\leq K \tilde E_{i_1,+}^{mix} .\end{aligned}\end{equation} 

To compare the crossing energies from (\ref{gp}), we will now estimate the double $\alpha$ and the double $\beta$ terms that arise in $S^{dbl}_{i_0,-}$, $S^{dbl}_{i_1,+}$, $\tilde S^{dbl}_{i_0,-}$ and $\tilde S^{dbl}_{i_1,+}$, by sums with mixed, $\alpha\beta$ terms. This is done in Lemma \ref{double term}. Lemma \ref{local_oscillations} gives us the tool that can be viewed as a ``Harnack inequality'' for crossing sequences. It gives us a local estimate on the difference of two solutions of (\ref{rr}). In fact, it tells us how we can estimate specific $\alpha$ terms by $\beta$ terms and vice versa. 

\begin{lemma}\label{local_oscillations}
It holds for all $i$ with $\beta_i=0$ that $$0\leq \left(\sum_{j=i-r}^{i}+\sum_{j=i}^{i+r}\right)(-\beta_j) \leq \frac{K}{\lambda} \left(\sum_{j=i-r}^{i}+\sum_{j=i}^{i+r}\right)\alpha_j$$ and similarly, for all $i$ with $\alpha_i=0$, it holds $$0\leq \left(\sum_{j=i-r}^{i}+\sum_{j=i}^{i+r}\right)\alpha_j \leq \frac{K}{\lambda} \left(\sum_{j=i-r}^{i}+\sum_{j=i}^{i+r}\right)(-\beta_j).$$
\end{lemma}

\begin{proof}
We only prove the first inequality in the lemma. The recurrence relation with interpolation gives as in (\ref{x-y}): \begin{align}\nonumber &0=\p_iW(y)-\p_iW(x)=\sum_{j=i-r}^{i}(\p_i S_j(y)-\p_i S_j(x))=\\ &\nonumber =\sum_{j=i-r}^{i}\int_0^1\p_{j,i}S_j(\tau y+(1-\tau)x)d \tau (y_j-x_j)+\sum_{j=i}^{i+r} \int_0^1\p_{j,i}S_i(\tau y+(1-\tau)x)d \tau (y_j-x_j).\end{align} Bringing the terms with $y_i-x_i=\alpha_i>0$ to the right-hand side of the equality, we get: \begin{align} \nonumber & -\sum_{j=i-r}^{i}\int_0^1\p_{j,i}S_j(\tau x+(1-\tau)y)d \tau \alpha_j-\sum_{j=i}^{i+r} \int_0^1\p_{j,i}S_i(\tau x+(1-\tau)y)d \tau \alpha_j=\\ \nonumber = &\sum_{j=i-r}^{i}\int_0^1\p_{j,i}S_j(\tau x+(1-\tau)y)d \tau \beta_j+\sum_{j=i}^{i+r} \int_0^1\p_{j,i}S_i(\tau x+(1-\tau)y)d \tau \beta_j.\end{align} Assuming that $\beta_i=0$, and since $\beta\leq 0$, it follows on one hand by the twist condition (\ref{strong twist}) that all the terms on the right-hand side of the equality are non-negative. On the other hand, the left-hand side can be estimated by the uniform bound on the second derivatives from definition \ref{S}, which gives $$K\left(\sum_{j=i-r}^{i}+\sum_{j=i}^{i+r}\right)\alpha_j \geq \lambda \left(\sum_{j=i-r}^{i}+\sum_{j=i}^{i+r}\right)(-\beta_j)\geq0.$$ 
\end{proof}

Let us set some notation before proceeding with Lemma \ref{double term}. Define for every $j\in \Z$ the indices $k(j)$ and $l(j)$ as \begin{equation}\label{beta k}\beta_{k(j)}:=\min\{\beta_i\ | \ i\in [j-r,j+r]\}\ \text{ and } \ \alpha_{l(j)}:=\max\{\alpha_i\ | \ i\in [j-r,j+r]\}\end{equation} as a largest $\beta$-term in $[j-r,j+r]$ and a largest $\alpha$-term in $[j-r,j+r]$, respectively. In case $k(j)$ or $l(j)$ are not unique, we may choose the smallest. For the sake of brevity, we define also $$c:=\frac{2K^2(2r+1)}{\lambda}.$$ Moreover, define for a domain $B=[i_0-r,i_1]$ the following quantities \begin{equation}\label{def-dbl}\begin{aligned}E_{i_0,-}^{dbl}&:=-\sum_{j={k(i_0)}-r}^{{k(i_0)+r}}\beta_{k(i_0)} \alpha_j \\ E_{i_1,+}^{dbl}&:=-\sum_{j=k(i_1)-r}^{k(i_1)+r}\beta_{k(i_1)} \alpha_j \\ 
\tilde E_{i_0,-}^{dbl}&:=-\sum_{j=l(i_0)-r}^{l(i_0)+r}\alpha_{l(i_0)} \beta_j \\ \tilde E_{i_1,+}^{dbl}&:=-\sum_{j=l(i_1)-r}^{l(i_1)+r}\alpha_{l(i_1)} \beta_j\end{aligned}\end{equation}

\begin{lemma}\label{double term}
Let $B:=[i_0-r,i_1]$ be such that $\alpha_{i_0}=\alpha_{i_1}=0$ and assume that $i_1-i_0>2 r$. Then the following estimates hold: $$ S_{i_0,-}^{dbl}\leq c E_{i_0,-}^{dbl}\ \ \text{and} \ \ S_{i_1,+}^{dbl}\leq cE_{i_1,+}^{dbl}.$$ Similarly, if $\beta_{i_0}=\beta_{i_1}=0$, then it holds: $$\tilde S_{i_0,-}^{dbl}\leq c \tilde E_{i_0,-}^{dbl}\ \ \text{and} \ \ \tilde S_{i_1,+}^{dbl}\leq c \tilde E_{i_1,+}^{dbl}.$$
\end{lemma}

\begin{proof}
We only explain how we can get the estimate for $S_{i_0,-}^{dbl}$, the other cases being analogous. Recall that

$$S_{i_0,-}^{dbl} :=\sum_{i=i_0-r}^{i_0-1}\sum_{j=i_0}^{i+r} I_B^{i,j}(x)\alpha_i\alpha_j\leq K \sum_{i=i_0-r}^{i_0-1}\sum_{j=i_0}^{i+r}\alpha_i\alpha_j.$$

Assume first that ${k(i_0)} \in [i_0-r,i_0]$, where $k(i_0)$ is as in (\ref{beta k}). Then, because $\alpha_{i_0}=0$, we can estimate the $\alpha_i\alpha_j$-terms around $i_0$ with Lemma \ref{local_oscillations}, by $$\sum_{j=i_0}^{i_0+r}\alpha_j\leq \left(\sum_{j=i_0-r}^{i_0-1}+\sum_{j=i_0}^{i_0+r}\right)\alpha_j \leq -\frac{K}{\lambda} \left(\sum_{j=i_0-r}^{i_0-1}+\sum_{j=i_0}^{i_0+r}\right)\beta_j\leq -\frac{K(2 r+1)}{\lambda}\beta_{k(i_0)}.$$ This implies \begin{equation}\label{dt1}\sum_{i=i_0-r}^{i_0-1}\sum_{j=i_0}^{i+r}\alpha_i\alpha_j \leq \left(\sum_{i=i_0-r}^{i_0-1}\alpha_i\right)\left(\sum_{j=i_0}^{i_0+r}\alpha_j\right)\leq-\frac{2K(2r+1)}{\lambda} \sum_{j=k(i_0)-r}^{k(i_0)+r}\beta_{k(i_0)}\alpha_j,\end{equation} where the last inequality follows because $\{i_0-r,...,i_0-1\}\subset \{k(i_0)-r,...,k(i_0)+r\}$. In case that ${k(i_0)} \in [i_0+1,i_0+r]$, we equivalently as above first get the estimate $$\sum_{j=i_0-r}^{i_0}\alpha_j\leq  -\frac{2K(2r+1)}{\lambda}\beta_{k(i_0)}$$ which similarly gives the inequality (\ref{dt1}). 
\end{proof}

Define for $B=[i_0-r,i_1]$ the boundary terms \begin{equation}\label{def-e}E_{i_0}^-:=E_{i_0,-}^{mix}+E_{i_0,-}^{dbl} \ \text{ and } \ E_{i_1}^+:=E_{i_1,+}^{mix}+E_{i_1,+}^{dbl},\end{equation} and similarly $\tilde E_{i_0}^-:=\tilde E_{i_0,-}^{mix}+\tilde E_{i_0,-}^{dbl}$ and $\tilde E_{i_1}^+:=\tilde E_{i_1,+}^{mix}+\tilde E_{i_1,+}^{dbl}$. By combining the definition of boundary energies in Proposition \ref{boundary energy}, (\ref{mixed estimate}) and Lemma \ref{double term}, we get an estimate for the boundary energies in terms of sums of finitely many mixed $\alpha_i\beta_j$ terms around $i_0$ and $i_1$. 

\begin{corollary}\label{swapping}
Let $B:=[i_0-r,i_1]$ be such that $\alpha_{i_0}=\alpha_{i_1}=0$ and assume that $i_1-i_0>2 r$. Then the following estimates hold: \begin{equation}\label{p}W_{i_0,-}^s\leq c E^-_{i_0} \ \text{ and } \ W_{i_1,+}^s\leq c E^+_{i_1}.\end{equation} 
Similarly, if $\beta_{i_0}=\beta_{i_1}=0$, it holds: \begin{equation}\label{tilde p}\tilde W_{i_0,-}^s\leq c \tilde E^-_{i_0} \ \text{ and } \ \tilde W_{i_1,+}^s\leq c \tilde E^+_{i_1}.\end{equation} 
\end{corollary}

\section{Bounded domains of crossings}\label{section bounded}

In this section we assume that two global minimizers $x,y \in \M$, have a bounded domain of crossing $D\neq \varnothing$. As explained in Section \ref{section min-max}, Corollary \ref{crosses} applies. In particular, we may assume without loss of generality that $x\leq y $ (or equivalently, $\beta=0$,) on $(-\infty,j_0-1]$ and $x\geq y$ (or equivalently, $\alpha=0$,) on $[j_1+1,\infty)$. I.e., we assume that $j_0:=\min\{i \in \Z \ | \ \beta_i<0\}$ and $j_1:=\max\{i \in \Z \ | \ \alpha_i>0\}$. A particular case of this situation arises when $x\in \B_\nu$, $y\in \B_{\rho}$ and $\rho \neq \nu$. Here it follows by the uniform estimates on Birkhoff sequences, see (\ref{rotation number estimate}), that $D$ is bounded. 

\begin{theorema}
Let $x,y \in \M$ be global minimizers and $D=[j_0,j_1]$ be a bounded domain of crossings for $x$ and $y$. Then the size of $D$ is uniformly bounded by $$|D|= j_1-j_0\leq\tilde{K}:=\lceil 12 r\lambda^{-2}c^2+3r\rceil, $$ where $c=\frac{2K^2(2r+1)}{\lambda}$ and where $\lceil \cdot \rceil$ denotes the ceiling function. \\
\end{theorema}

\begin{proof}
We follow a proof by contradiction and assume that $j_1-j_0>\lceil 12 r\lambda^{-2}c^2+3r\rceil $. 

Define $B:=[j_0-r,j_1+r]$, so that $M^B(x)|_{[j_1+1,j_1+r]} \equiv x|_{[j_1+1,j_1+r]}$, since $x\geq y$ on $[j_1+1,\infty)$ by assumption. This implies that $\alpha^B(x)|_{[j_0,\infty)}=0$ and in particular, $W_{j_1+r,+}^b=0$ so that $W_B^c(M^B(y),m^B(x))=W_{j_0,-}^b$. By the general principle of the proofs (\ref{gp}) it must hold that $W_{j_0,-}^b\geq W_B^c(x,y)$. Since $j_0=\min\{i\in \Z\ | \ \beta_i>0\}$, it follows that $\alpha_{j_0}=0$, so we can apply Corollary \ref{swapping} to obtain $cE_{j_0}^- \geq W_{j_0,-}^b\geq W_B^c(x,y)$. If we use (\ref{min-max}) to estimate $W_B^c(x,y)$, it must hold that \begin{equation}\label{eq1}cE_{j_0}^- \geq -\lambda \sum_{i=j_0-r}^{j_1+r} \sum_{j=i-r}^{i+r} (\alpha_j\beta_i+\alpha_i\beta_j).\end{equation} 

The right side of (\ref{eq1}) can be estimated in the following way: by Corollary \ref{dense_crossings}, there is a finite sequence $i_n \in [j_0+2r,j_1-r]$ with $\alpha_{i_n}>0$ (which implies that $\beta_{i_n}=0$) and such that $2 r< i_n-i_{n+1}\leq 3 r$.  It holds for all $n$ that $l(i_n)\neq l(i_{n+1})$, where $l(i)$ is as defined in (\ref{beta k}), so the supports of $\tilde E_{i_n,+}^{dbl}$ are disjoint for all $n$. Moreover, the supports of $\tilde E_{i_n,+}^{mix}$ are also  disjoint for all $n$, so it holds for $\tilde E_{i_n}^+=\tilde E_{i_n,+}^{dbl}+\tilde E_{i_n,+}^{mix}$ that $$-2 \sum_{i=j_0-r}^{j_1+r} \sum_{j=i-r}^{i+r} (\alpha_j\beta_i+\alpha_i\beta_j) > \sum_{n=1}^{N} \tilde E_{i_n}^+.$$ By Corollary \ref{dense_crossings}, it holds for all $n$ that $\tilde E_{i_n}^+>0$, so also $0<\tilde E_{i_{\bar n}}^+:=\min_{n\in[1,N]}E_{i_{n}}^+$ for which \begin{equation}\label{eq2}-2\sum_{i=j_0-r}^{j_1+r} \sum_{j=i-r}^{i} (\alpha_j\beta_i+\alpha_i\beta_j)> \sum_{n=1}^{N}\tilde E_{i_n}^+\geq N \tilde E_{i_{\bar n}}^+.\end{equation} 

Since $j_1-j_0>\lceil 12 r\lambda^{-2}c^2+3r\rceil $, it holds that $N>\lceil 4\lambda^{-2}c^2\rceil $. Putting (\ref{eq1}) and (\ref{eq2}) together and using the fact that $N>\lceil 4\lambda^{-2}c^2\rceil $, it follows that \begin{equation}\label{eq3}\frac{\lambda}{2} E_{j_0}^-  > c \tilde E_{i_{\bar n}}^+.\end{equation} 

This brings us to the second part of the proof. Define $\tilde{B}:=[j_0-2r,i_{\bar n}]$ and observe that it holds for $W_{\tilde B}^c(M^{\tilde B}(y),m^{\tilde{B}}(x))=\tilde{W}_{j_0-r,-}^b+\tilde{W}_{i_{\bar n},+}^b$ that $\tilde{W}_{j_0-r,-}^b=0$ (by the same reasoning which confirmed that $W_{j_1+r,+}^b=0$ at the beginning of the proof). Since $\{i_n\}_{n=1}^N\subset[j_0+2r,j_1-r]$ it holds in particular that $j_0+2r \leq i_{\bar n}+r$. This implies that $[j_0- 2r,j_0+2 r] \subset \tilde B$ and we can estimate the crossing energy $W^c_{\tilde B}(x,y)$ by the boundary energy $E_{j_0}^-$ in the following way: \begin{align*} W^c_{\tilde B}(x,y)&\geq  -\lambda\left(\sum_{i=j_0-r}^{j_0-1}\sum_{j=i}^{i+r}\alpha_i\beta_j + \sum_{i=j_0-r}^{j_0-1}\sum_{j=i}^{j_0-1}\beta_i\alpha_j\right)=\lambda E_{j_0,-}^{mix},\\ W^c_{\tilde B}(x,y)&\geq  -\lambda \sum_{j={k(j_0)}-r}^{{k(j_0)}+r}\beta_{k(j_0)} \alpha_j= \lambda E_{j_0,-}^{dbl},\end{align*} where we used definitions (\ref{def-mix}) and (\ref{def-dbl}). Together, these two inequalities show that\begin{equation}W^c_{\tilde B}(x,y)\geq \frac{\lambda}{2}(E_{j_0,-}^{mix}+E_{\p B_-}^{dbl})=\frac{\lambda}{2} E_{j_0}^-.\end{equation} Combining this estimate with the inequality (\ref{eq3}) above and using Corollary \ref{swapping}, with the fact that $\beta_{i_{\bar n}}=0$, it follows that $$W^c_{\tilde B}(x,y)> c \tilde E^+_{i_{\bar n}} \geq \tilde{W}_{i_{\bar n},+}^b.$$ Since $\tilde{W}_{j_0-r,-}^b=0$, it follows that $$W_{\tilde{B}}^c(x,y)>\tilde{W}_{j_0-r,-}^b+ \tilde{W}_{i_{\bar n},+}^b= W_{\tilde{B}}^c(M^{\tilde{B}}(y),m^{\tilde{B}}(x)),$$ a contradiction to the general principle of the proof (\ref{gp}). 

So, it must hold that $j_1-j_0\leq\lceil 12 r\lambda^{-2}c^2+3r\rceil $.
\end{proof}

\section{Unbounded domains of crossings}\label{section unbounded domain}

In this section we assume that the domain of crossing $D$ for global minimizers $x$ and $y$ is a connected unbounded domain. So, $D=[j_0,\infty)$, $D=(-\infty,j_0]$ or $D=(-\infty,+\infty)$. The ideas in the proofs in this section are in many ways similar to that of Theorem A.

\begin{theoremb}
Assume that the domain of crossing $D$ for $x,y\in \M$ is infinite. Then there is a constant $d \in \N$ that depends only on the range of interaction $r$ and the uniform constants $\lambda$ and $K$ from Definition \ref{S}, such that the following holds. There exist monotone sequences $k_n,l_n \in D$ with $|k_{n+1}-k_n|\leq d$ and $|l_n-k_n|\leq r$ which satisfy $$\text{ for all } n, \  x_{k_n}>y_{k_n}, \ x_{l_n}<y_{l_n}\ \text{ and }  \ (x_{k_n}-y_{k_n})(y_{l_n}-x_{l_n})\geq 2^n.$$ The explicit expression for $d$ is $$d:=6r\lceil 24K^2 (2r+1)r^2\lambda^{-2}\rceil+4r.$$
\end{theoremb}

We split the proof of Theorem B into two cases, covered in Theorem B1 and Theorem B2. As explained in Section \ref{section min-max}, if the domain of crossing $D$ is unbounded, then Proposition \ref{infinite_oscillations} holds. Explicitly, we may take an infinite sub-domain $\tilde D\subset D$, such that there exists no segment $I\subset \tilde D$ with $|I|\geq r$ and such that $\alpha|_I\equiv 0$ or $\beta|_I\equiv 0$. Theorem B1 applies to the case where $\tilde D\neq \Z$.

\begin{theoremb1}
Assume that the global minimizers $x$ and $y$ are crossing in an unbounded domain $D$, such that it holds for $\tilde D$ from Proposition \ref{infinite_oscillations} that $\tilde D\neq \Z$. Then there is a constant $d \in \N$ and two monotone infinite sequences $k_n,l_n\in D$, such that $|l_n-k_n|\leq r$ and $|k_{n+1}-k_n|\leq d$, with the following property: $$-\alpha_{k_n}\beta_{l_n}\geq 2^n.$$ The explicit expression for $d$ is $$d:= 6r\lceil 12 c  r^2\lambda^{-1}\rceil+4r,$$ where $c=\frac{2 K^2 (2r +1)}{\lambda}$ and $K$ and $\lambda$ are the uniform constants from Definition \ref{S}.
\end{theoremb1}

\begin{proof}
Without loss of generality, we may assume that $\tilde D=[k_0,\infty)$, for some $k_0\in \Z$. The case where $\tilde D=(-\infty,k_0]$ then follows by applying the map $-Id$ on $\Z$. Furthermore, we may assume that $x>y$ on $[k_0-r, k_0-1]$. This implies that $W_{k_0,-}^b=0$, because then $\alpha\equiv 0$ on $[k_0- r, k_0-1]$ (see Proposition \ref{boundary energy}). We can recover the case where $y>x$ by swapping the notation for $x$ and $y$. 

\textbf{Part 1 of the proof:}\\
By Lemma \ref{even_crossings}, there exists an infinite monotone sequence $\{j_n\}_{n\in \N\cup 0} \subset \tilde D$, such that $\alpha_{j_n}=0$ for all $n$, $j_0=k_0$ and $2r< j_{n+1}-j_n\leq 3r$ for all $n$. Notice that we have quite a lot of freedom in choosing this sequence. Moreover, for all $n\in \N$ it holds that $k(j_n)$ are distinct, where $k(i)$ is defined as in (\ref{beta k}). This implies that the supports of $E^+_{j_n}$, for different $j_n$, are disjoint.

Let $c=\frac{2 K^2 (2r +1)}{\lambda}$ as in Lemma \ref{double term} and define $N:=\lceil 12 c r^2\lambda^{-1}\rceil $. Define for every $m>1$ the domain $B^m:=[k_0-r,j_m]\subset \tilde D$. Then it holds for every $m>N$ that the finite subsequence $\{j_n\}_{n=1}^{N} \subset [k_0-r,j_m-2r]$. By definition of $B^m$ one of the boundary energies is $W_{k_0,-}^b=0$ and by the general principle of the proof (\ref{gp}) and Corollary \ref{swapping} the following inequalities need to be satisfied: \begin{equation}\label{inun1}cE_{j_m}^+\geq W^b_{j_m,+}\geq W_{B^m}^c(x,y)\geq \frac{\lambda}{2}\sum_{n=1}^NE_{j_n}^+.\end{equation} As in the proof of Theorem A, we now choose $j_{n_1}\in \{j_n \ | \ 1\leq n \leq N\}$ such that $$E_{j_{n_1}}^+:=\min_{n=1,...,N}E_{j_n}^+>0.$$ This implies that $c E_{j_m}^+\geq \frac{N \lambda}{2}E^+_{j_{n_1}}$ and since $N\geq 12c r^2\lambda^{-1}$, it follows for all $m>N$ that \begin{equation}\label{inun2}E_{j_m}^+\geq 6 r^2 E^+_{j_{n_1}}.\end{equation} 

Now we construct $j_{n_2}$. Observe that if $m>2N$ it holds for the finite sub-sequence $\{j_n\}_{n=N+1}^{2N}$ that it lies in $B^m$. As in (\ref{inun1}) we observe by the general principle (\ref{gp}) that for all $m>2N$, $$cE_{j_m}^+\geq\frac{\lambda}{2}\sum_{n=N+1}^{2N}E_{j_n}^+.$$ Define now $$E_{j_{n_2}}^+:=\min_{n=N+1,...,2N}E_{j_n}^+\geq 6 r^2 E^+_{j_{n_1}}$$ which similarly as in (\ref{inun2}) gives us for all $m>2N$ the inequality \begin{equation}\label{inun3}E_{j_m}^+\geq 6 r^2 E^+_{j_{n_2}}.\end{equation} 

Inductively repeating this procedure gives us the infinite monotone sub-sequence $\{j_{n_k}\}_{k\in \N}$ with \begin{equation}\label{inun3}E_{j_{n_k}}^+:=\min_{n=(k-1)N+1,...,kN}E_{j_n}^+\geq 6  r^2 E^+_{j_{n_{k-1}}}.\end{equation}

\textbf{Part 2 of the proof:}\\
In this part of the proof we will isolate from each $E_{j_{n_k}}^+$ from part 1 of the proof a specific pair $\alpha_i,\beta_j$. The corresponding sequences of indices will satisfy the statements of the Theorem. Recall by (\ref{def-mix}), (\ref{def-dbl}) and (\ref{def-e}) that $E_{k}^+$ is defined as a sum of finitely many $\alpha_i\beta_j$ terms with $i,j\in [k-2r, k+2r]$. We denote $$\max\hspace{-0.05cm}^+(k):=\max\{|\alpha_i\beta_j| \ | \ \{i,j\} \ \text{ such that }\ \alpha_i\beta_j \ \text{ appears in the definition of } \ E_k^+\}.$$ Then it holds by (\ref{def-mix}) and (\ref{def-dbl}) that $E_{k,+}^{dbl}\leq (2r+1) \max^+(k)$ and $E_{k,+}^{mix}\leq 2  r^2 \max^+(k)$, so it holds since $r\geq 2$ that \begin{equation}\label{inun4}3  r^2 \max\hspace{-0.05cm}^+(k)\geq E_{k}^+\geq \max\hspace{-0.05cm}^+(k).\end{equation}

Combining (\ref{inun3}) and (\ref{inun4}) implies that $\max^+(j_{n_k})\geq 2 \max^+(j_{n_{k-1}})$ for all $k\in \N$. Let $\alpha_{k_n}\beta_{l_n}:=\max^+(j_{n})$ and note that $j_{n_k}-j_{n_{k-1}}\leq 2N 3 r$. After reindexing, this gives us the sequences $\{\alpha_{k_n}\}_{n\in \N}$ and $\{\beta_{l_n}\}_{n\in \N}$ such that $$\alpha_{k_n}\beta_{l_n}-\alpha_{k_{n-1}}\beta_{l_{n-1}}\leq 6r\lceil 12 c  r^2\lambda^{-1}\rceil+4r$$ and $\alpha_{k_n}\beta_{l_n}\geq 2^n$, which finishes the proof.
\end{proof}

Theorem B2 applies to the case of $\tilde D=\Z$, where $\tilde D$ is as in Proposition \ref{infinite_oscillations}. The statement of Theorem B2 is the same as the statement of Theorem B1, but the proof of Theorem B2 is slightly different, so we present it separately.

\begin{theoremb2}
Assume that the global minimizers $x$ and $y$ are crossing in an unbounded domain $D$, such that it holds for $\tilde D$ from Proposition \ref{infinite_oscillations} that $\tilde D= \Z$. Then there is a constant $d \in \N$ and monotone infinite sequences $k_n,l_n\in D$, such that $|l_n-k_n|\leq r$ and $|k_{n+1}-k_n|\leq d$, with the following property: $$-\alpha_{k_n}\beta_{l_n}\geq 2^n.$$ The explicit expression for $d$ is the same as in Theorem B1, $$d:=8r\lceil 12 c r^2\lambda^{-1}\rceil+12r,$$ where $c=\frac{2 K^2 (2r +1)}{\lambda}$ and $K$ and $\lambda$ are the uniform constants from Definition \ref{S}.
\end{theoremb2}

\begin{proof}
Similarly as in the proof of Theorem B1, there exists, by Lemma \ref{even_crossings}, a \textit{bi}-infinite monotone sequence $\{j_n\}_{n\in \Z} \subset \tilde D$, such that $\alpha_{j_n}=0$ for all $n$ and $2r< j_{n+1}-j_n\leq 3r$ for all $n$. Then it holds for all $n\in \Z$ that $k(j_n)$ are distinct, where $k(i)$ is defined as in (\ref{beta k}). This implies that the supports of $E^+_{j_n}$ for different $n$, are disjoint. Also, the supports of $E^-_{j_n}$ for different $n$ are disjoint.

Let $c=\frac{2 K^2 (2r +1)}{\lambda}$ as in Lemma \ref{double term} and define $N:=\lceil 12 c r^2\lambda^{-1}\rceil $. Define for every two integers $\tilde m>m$ the domain $B^{m,\tilde m}:=[j_m,j_{\tilde m}]$ . Then it holds for every $m,p>N$ that $\{j_n\}_{n=-N}^N\subset B^{-m,p}$. By the general principle of the proof \ref{gp}, it has to holds that  \begin{equation}\label{2inun1}c(E_{j_{-m}}^-+E_{j_p}^+)\geq W^b_{j_{-m},-}+W^b_{j_p,+}\geq W_{B^{-m,p}}^c(x,y)\geq \frac{\lambda}{2}\left(\sum_{n=1}^NE_{j_{-n}}^-+\sum_{n=1}^NE_{j_n}^+\right).\end{equation}

As in the proof of Theorem A and Theorem B1, we now choose $j_{n_{-1}}\in \{j_n\}_{n=-1}^{-N}$ such that $$E_{j_{n_{-1}}}^-:=\min_{n=-1,...,-N}E_{j_n}^->0.$$ Moreover, we choose $j_{n_1}\in \{j_n\}_{n=1}^N$ such that $$E_{j_{n_1}}^+:=\min_{n=1,...,N}E_{j_n}^+>0.$$ Then it holds by (\ref{2inun1}) for every $m,p>N$ that $$c (E_{j_{-m}}^-+E_{j_p}^+)\geq \frac{N\lambda}{2}(E_{j_{n_{-1}}}^-+E_{j_{n_1}}^+).$$ Plugging in the definition of $N$, we arrive to the following: for every $p,m>N$ it must hold that \begin{equation}\label{2inun2}E_{j_{-m}}^-+E_{j_p}^+\geq 6 r^2(E_{j_{n_{-1}}}^-+E_{j_{n_1}}^+). \end{equation} 

Since $E_{j_n}^{\pm}>0$ for all $n$ it follows from (\ref{2inun2}) that one of the following three cases must hold.\\
\textbf{Case 1:} there exists an $m_0>N$ such that $E_{j_{-m_0}}^-< 6 r^2E_{j_{n_{-1}}}^-$. In this case it must hold for all $p>N$ \begin{equation}\label{2inun3a} E_{j_{-m_0}}^-< 6 r^2E_{j_{n_{-1}}}^- \ \text{ and } \ E_{j_p}^+\geq 6 r^2E_{j_{n_1}}^+.\end{equation}
\textbf{Case 2:} there exists a $p_0>N$ such that $E_{j_p}^+< 6 r^2E_{j_{n_1}}^+$. In this case it must hold for all $m>N$ \begin{equation}\label{2inun3b} E_{j_{-m}}^-\geq 6 r^2E_{j_{n_{-1}}}^- \ \text{ and } \ E_{j_{p_0}}^+< 6 r^2E_{j_{n_1}}^+.\end{equation}
\textbf{Case 3:} for all $m,p>N$, it holds that \begin{equation}\label{2inun3c} E_{j_{-m}}^-\geq 6  r^2E_{j_{n_{-1}}}^- \  \text{ and } \ E_{j_p}^+\geq 6 r^2E_{j_{n_1}}^+. \end{equation} 

We construct the second element of the subsequence $\{j_{n_k}\}_{k\in \N}$, i.e. $j_{n_2}$, for each of the cases above. Keep in mind that we want $\{j_{n_k}\}_{k\in \N}$ to be a monotone infinite sequence and not a bi-infinite sequence in $\tilde D$. \\
\textbf{Case 1:} define $j_{n_2}\in B^{-m_0,2N}$ by $$E_{j_{n_2}}^+:=\min_{n=N+1,...,2N}E_{j_n}^+\geq6 r^2E_{j_{n_1}}^+.$$ Similarly as for (\ref{2inun2}), this leads for every $m>N, p>2N$ to the inequality $$E_{j_{-m}}^-+E_{j_p}^+\geq 6 r^2(E_{j_{n_{-1}}}^-+E_{j_{n_2}}^+) $$ and since $E_{j_{-m_0}}^-< 6 r^2E_{j_{n_{-1}}}^-$ it follows for all $p>2N$ that \begin{equation}\label{2inun4a} E_{j_{-m_0}}^-< 6 r^2E_{j_{n_{-1}}}^- \ \text{ and } \ E_{j_p}^+\geq 6 r^2E_{j_{n_2}}^+.\end{equation} Continuing this procedure inductively leads to a monotone increasing sequence $\{j_{n_k}\}_{k\in \N}$ where $$E_{j_{n_k}}^+:= \min_{n=(k-1)N+1,...,kN}E_{j_n}^+\geq 6  r^2 E_{j_{n_{k-1}}}.$$
\textbf{Case 2:} define $j_{n_{-2}}\in B^{-2N,p_0}$ by $$E_{j_{n_{-2}}}^-:=\min_{n=-N-1,...,-2N}E_{j_n}^+\geq6 r^2E_{j_{n_1}}^+.$$ Similarly as for Case 1, it follows for all $m>2N$ that \begin{equation}\label{2inun4b} E_{j_{-m}}^-\geq 6 r^2E_{j_{n_{-2}}}^- \ \text{ and } \ E_{j_{p_0}}^+< 6 r^2E_{j_{n_1}}^+.\end{equation} Continuing this procedure inductively leads to a monotone increasing sequence $\{j_{n_{-k}}\}_{k\in \N}$ where $$E_{j_{n_{-k}}}^-:= \min_{n=(-k+1)N+1,...,-kN}E_{j_n}^-\geq 6  r^2 E_{j_{n_{-k+1}}}^-.$$
\textbf{Case 3:} define $j_{n_{-2}},j_{n_2}\in B^{-2N,2N}$ by $$E_{j_{n_{-2}}}^-:=\min_{n=-N-1,...,-2N}E_{j_n}^-\geq 6 r^2E_{j_{n_{-1}}}^- \ \text{ and } \ E_{j_{n_2}}^+:=\min_{n=N+1,...,2N}E_{j_n}^+\geq6 r^2E_{j_{n_1}}^+.$$ Similarly as for (\ref{2inun2}), this leads for every $m,p>2N$ to the inequality \begin{equation}\label{2inun4c}E_{j_{-m}}^-+E_{j_p}^+\geq 6 r^2(E_{j_{n_{-2}}}^-+E_{j_{n_2}}^+). \end{equation} Obviously, (\ref{2inun4c}) again implies one of the cases 1-3, with the accompanying inequalities corresponding to (\ref{2inun3a}),(\ref{2inun3b}) and (\ref{2inun3c}). Inductively proceeding, it can happen that we end up with case 3 for every step and obtain a bi-infinite monotone sequence $\{j_{n_k}\}_{k\in \Z\backslash 0}$ such that both $E_{j_{n_{-k}}}\geq 6  r^2 E_{j_{n_{-k+1}}}$ and $E_{j_{n_k}}\geq 6  r^2 E_{j_{n_{k-1}}}$ holds. If, on the other hand, either case 1 or case 2 applies, at some step of the induction, this gives us an infinite monotone increasing, or an infinite monotone decreasing sequence, respectively. This finishes the proof of case 3.

The rest of the proof is exactly the same as part 2 of the proof of Theorem B1.
\end{proof}

Note that the constant $d$ in Theorem B does not depend on the sequences $x$ and $y$. We think that $d$ is not optimal, however it gives a qualitative estimate on the growth rate of the oscillations for the difference $x-y$.

\section{A dichotomy theorem}\label{section birkhoff}

Recall the definition of a Birkhoff sequence:  $x\in \B$ if for all $k,l\in \Z\times\Z$ either $\tau_{k,l}x\geq x$ or $\tau_{k,l}x\leq x$. Moreover, recall from Section \ref{standard results} that Birkhoff sequences have a well defined rotation number $\rho(x):=\lim_{n\to \infty}\frac{x_n}{n}\in \R$, for which the following uniform estimate is satisfied: $|x_n-x_0-\rho(x)n|\leq 1$. In this section we prove the Dichotomy Theorem announced in the introduction. It states that every global minimizer is either Birkhoff, or grows exponentially and oscillates. This is an application of Theorem A and Theorem B to $x$ and $\tau_{k,l}x=y$.

\begin{definition}\label{almost Birkhoff}
Let us call a global minimizer $x\in \M$ \textit{almost Birkhoff}, if for all $k,l\in\Z\times\Z$ the domain of crossing $D$ for $x$ and $\tau_{k,l}x$ is finite. Denote the set of almost Birkhoff global minimizers by $\mathcal{ABM}$.
\end{definition}

By Theorem A, for any $x\in \mathcal{ABM}$ and for any $k,l\in \Z\times\Z$, the domain of crossing $D$ for $x$ and $\tau_{k,l}x$ has size $|D|\leq \tilde K$, independent of $k$ and $l$. Moreover, if $|D|>0$, then $D=[j_0,j_1]$ for some $j_1-j_0<\tilde{K}$, and it holds for all $i<j_0$ and $j>j_1$ that $(x_i-y_i)(x_j-y_j)<0$.

It is clear that Birkhoff global minimizers are almost Birkhoff global minimizers. The main result of this section is, that all almost Birkhoff global minimizers are Birkhoff. This implies that $\mathcal{ABM} = \BM$. We closely follow the ideas from \cite{MatherForni}. The following lemma is well known for classical Aubry-Mather Theory, see for example \cite{MatherForni}, $\S 14$, `Addendum to Aubry's Lemma'.

\begin{lemma}\label{addendum}
Let $x,y\in \M$ be such that their domain of crossing $D$ is finite and assume that $x$ and $y$ are asymptotic, i.e. that $|x_i - y_i| \to 0$ for $i\to \infty$ or for $i \to -\infty$. Then $x\geq y$ or $y\geq x$, or equivalently, $D=\varnothing$.
\end{lemma}

\begin{proof}
Assume not, i.e. $D\neq \varnothing$. Since $D$ is finite, we may assume that there are indices $j_0,j_1$ such that $x_i\leq y_i$ for all $i<j_0$ and $x_i\geq y_i$ for all $i>j_1$. By Theorem A it follows that $0<j_1-j_0\leq \tilde K$. This implies by Lemma \ref{min-max lemma} and in particular by (\ref{min-max}) that for any finite $B=[i_0,i_1]\subset \Z$ with $j_0,j_1 \in \mathring{B}$, $W^c_B(x,y)>0$. Assume that $y_i-x_i \to 0$ for $i\to -\infty$. Recall that by the general principle (\ref{gp}) and by Proposition \ref{boundary energy}, it must hold for any finite $B=[i_0,i_1]\subset \Z$ that $$W_B^c(x,y)\leq W_B^c(M^B(x),m^B(y))=W^b_{i_0,-}+W^b_{i_1,+}.$$ Choose a domain $B:=[i_0,i_1]$ with $i_1\geq j_1+ r$ it follows that $W^b_{i_1,+}=0$. Because $y_i-x_i \to 0$ for $i\to -\infty$, it moreover follows that for every $\varepsilon>0$, there is a $k<j_0$ such that for all $i_0<k$, $W^b_{i_0,-}<\varepsilon$. This implies that for every $\varepsilon>0$ there is a large enough $B$ such that $W_B^c(x,y)<\varepsilon$. Since $W_{\tilde B}^c(x,y)\leq W_{B}^c(x,y)$ if $\tilde B \subset B$, it follows that for every $B$, $W_B^c(x,y)=0$, a contradiction that finishes the proof.
\end{proof}

As in $\S11$ of \cite{MatherForni}, we introduce the following asymptotic ordering relations. 

\begin{definition} We define the relations $>_\alpha$, $>_\omega$ by saying that $x>_\alpha y$ if there is an $i_0\in \Z$ such that $x_i>y_i$ for all $i\leq i_0$ and $x>_\omega y$ when $x_i>y_i$ for all $i\geq j_0$, for some $j_0 \in \Z$. Analogously, define also $<_\alpha$ and $<_\omega$. 
\end{definition}

The following proposition is clear from Definition \ref{almost Birkhoff}. 

\begin{proposition}\label{trichotomy} It holds for every $x  \in \mathcal{ABM}$ and every $k,l\in \Z\times \Z$ that either $x$ and $\tau_{k,l}x$ are ordered ($x\geq \tau_{k,l}x$ or $x\leq \tau_{k,l}x$), or either \begin{equation}\label{not-ordered} (x>_\omega \tau_{k,l}x  \ \text{and}  \ x<_\alpha \tau_{k,l}x)  \ \text{or} \ ( x<_\omega \tau_{k,l}x \ \text{and}\ x>_\alpha \tau_{k,l}x).\end{equation}
\end{proposition}

In the following, for any $x \in \mathcal{ABM}$ an adapted definition of the rotation number $\tilde \rho(x)$ is introduced, which in the end turns out to be equivalent to the definition $\rho(x):=\lim_{n\to \infty}\frac{x_n}{n}\in \R$ from above. 

We recap the proof of the following Lemma from \cite{MatherForni} $\S11$.

\begin{lemma}\label{bla}
For every $x\in \mathcal{ABM}$, it holds that $\tau_{k,l}x>_\alpha x$, if and only if $\tau_{nk,nl}x>_\alpha x$ for all $n\in \N_+$.
\end{lemma} 

\begin{proof} First, it is clear that if $\tau_{k,l}x>_\alpha x$, then also $\tau_{(n+1)k,(n+1)l}x>_\alpha \tau_{nk, nl}x$ for all $n \in \N_+$, so $\tau_{nk,nl}x>_\alpha x$. 

On the other hand, if $\tau_{k,l}x\ngtr_\alpha x$, then by Proposition \ref{trichotomy} either $\tau_{k,l}x\leq x$ or $\tau_{k,l}x>_\omega x$. The first relation implies that for all $n\in \N_+$, $\tau_{nk,nl}x\leq x$. The second asymptotic relation implies that for all $n \in \N_+$, $\tau_{(n+1)k,(n+1)l}x>_\omega \tau_{nk,nl}x$, which in turn implies that $\tau_{nk,nl}x>_\omega x$, so $\tau_{nk,nl}x\ngtr_\alpha x$. 
\end{proof}

Lemma \ref{bla} has the following implication. Assume that $\frac{l'}{k'}>\frac{l}{k}$ (or equivalently $l' k> k' l$), and $\tau_{k,l}x >_\alpha x$. Then also $\tau_{k' k, k' l}x>_\alpha x$, so $\tau_{k' k, l' k}x >_\alpha x$ which implies that $\tau_{k', l'}x>_\alpha x$. Similarly, if $\frac{l'}{k'}>\frac{l}{k}$ and $\tau_{k,l}x >_\omega x$, then also $\tau_{k',l'}x >_\omega x$. Moreover, if $\frac{l'}{k'}<\frac{l}{k}$ and $\tau_{k,l}x <_{\alpha,\omega} x$, then also $\tau_{k',l'}x <_{\alpha,\omega} x$.

Now we define $$\rho_\alpha(x):=\inf \left\{\frac{l}{k}\ | \ \tau_{k,l}x>_\alpha x\right\}.$$ Because of Proposition \ref{trichotomy}, it holds that $\rho_{\alpha}(x)=\sup\left\{\frac{l}{k}\ | \ \tau_{k,l}x<_\alpha x\right\}$. Similarly, define $$\rho_\omega(x):=\inf\left\{\frac{l}{k}\ | \ \tau_{k,l}x>_\omega x\right\}=\sup\left\{\frac{l}{k}\ | \ \tau_{k,l}x<_\omega x\right\}.$$ 

\begin{proposition}\label{rot_number}
For every $x \in \mathcal{ABM}$, the number $$\tilde \rho(x):=\inf\left\{\frac{l}{k}\ | \ \tau_{k,l}x> x\right\}=\sup\left\{\frac{l}{k}\ | \ \tau_{k,l}x< x\right\} \in \R$$ is well defined.
\end{proposition}

\begin{proof}
First we show that $\rho_\alpha(x)=\rho_{\omega}(x)$. Assume that for $x \in \mathcal{ABM}$ there exists a $\frac{q}{p} \in \Q$, such that $\tau_{p,q}x>_\alpha x$ and $\tau_{p,q}x <_\omega x$. Then $\rho_\alpha (x)\leq\frac{q}{p}\leq\rho_\omega(x)$. On the other hand, it is easy to see that $\tau_{-p,-q}x<_\alpha x$ and $\tau_{-p,-q}x >_\omega x$ must hold, so $$\rho_\omega(x)=\inf\left\{\frac{l}{k}\ | \ \tau_{k,l}x>_\omega x\right\} \leq\frac{-q}{-p}=\frac{q}{p}\leq\sup\left\{\frac{l}{k}\ | \ \tau_{k,l}x<_\alpha x\right\}=\rho_\alpha(x).$$ This implies that for all $k,l$ with $\frac{l}{k}>\frac{q}{p}$ both $\tau_{k,l}x>_\omega x$ and $\tau_{k,l}x>_\alpha x$, so $\tau_{k,l}x> x$. I.e., for every $x \in \mathcal{ABM}$ \begin{equation}\label{rot_number}\rho_\alpha(x)=\rho_\omega(x)=\inf\left\{\frac{l}{k}\ | \ \tau_{k,l}x> x\right\}=\sup\left\{\frac{l}{k}\ | \ \tau_{k,l}x< x\right\}=:\tilde \rho(x).\end{equation}

We want to show that $\tilde \rho(x)\neq \infty$, by a slight modification of Theorem $11.2$ in \cite{MatherForni} which makes use of a proof by contradiction. So, let us assume that $\tilde \rho(x)=\infty$. Recall from the introduction, that periodic minimizers of all periods exist and that they are Birkhoff. Hence, we may choose a periodic minimizer $y \in \M_{1,q}$ such that $x_0>y_0$ and $x_{\tilde K}<y_{\tilde K}$, by choosing $q$ large enough, where $\tilde K$ is as in Theorem A. By definition of the rotation number it then holds that $\tau_{1,q+1}x< x$, so it holds for all $i$ that $x_{i+1}>x_i+q+1$. On the other hand, $\tau_{1,q+1}y=y+1$, so $y_{i+1}=y_i+q$. Hence, there is a integer $i'$, such that for all $i>i'$, $x_i>y_i$ holds. A similar consideration with $\tau_{-1,-q+1}$ shows that there is an integer $i''$, such that for all $i<i''$, $x_i<y_i$ must hold. But then the domain of crossing for $x$ and $y$ is finite and larger than $\tilde K$, which is a contradiction to Theorem A.

A similar argument shows that $\tilde \rho(x)\neq -\infty$.
\end{proof}

The following remark is a well known property of the rotation number, so we state it without proof (see e.g. \cite{gole01} or \cite{ghost_circles}).

\begin{remark}
Let $x\in \B$. Then $\rho(x)=\omega$ if and only if it holds for all $k,l\in \Z$ such that $\frac{l}{k}<\omega$, that $\tau_{k,l}x<x $, and for all $k,l\in \Z$ such that $\frac{l}{k}>\omega$, that $\tau_{k,l}x>x $. That is, $\rho(x)=\tilde \rho(x)$.
\end{remark}

Now we are set to prove the main result of this section:

\begin{theorem}\label{birkhoff}
If a global minimizer $x$ is almost Birkhoff, it is Birkhoff. In notation, $\mathcal{ABM}= \BM$.
\end{theorem}

\begin{proof}
We already proved that every $x \in \mathcal{ABM}$ has a corresponding rotation number $\rho(x):=\rho_\alpha(x)=\rho_\omega(x) \in \R$. If $\rho(x) \in \R\backslash \Q$, it holds for all $\frac{l}{k}\in \Q$ that $\tau_{k,l}x < x$ if $\frac{l}{k}<\rho(x)$ and $\tau_{k,l}x>x$ if $\frac{l}{k}>\rho(x)$ which shows that $x$ is Birkhoff. 

If $\rho(x)=\frac{q}{p}\in \Q$, the same relations as above hold for all $\frac{l}{k} \in \Q \backslash \{\frac{q}{p}\}$, so we only have to consider the behavior of $\tau_{p,q}x$. The following is also explained in the beginning of $\S13$ in \cite{MatherForni}, but for completeness we provide the necessary proofs. 

We start by proving the following claim. For any $x,y \in \mathcal{ABM}$ with $\rho(x)<\rho(y)$ it holds that $x>_\alpha y$ and $y>_\omega x$. We can easily see this by taking rational numbers $\rho(x)<\frac{l}{k}<\frac{l'}{k'}<\rho(y)$ for which it holds by definition that $\tau_{k',l'}y< y$ and $\tau_{k,l}x>x$ and that $k'l<kl'$ if $k>0$ and $k'>0$. It follows that  $$\tau_{k'k,0}(x-y)=\tau_{k'k,k'l}x-k'l-\tau_{k'k,l'k}y+kl'\geq x-y+1,$$ so the shift $\tau_{k'k,0}$ to the right increases the difference between $x$ and $y$, which proves the claim.

Assume now that $\tau_{p,q}x>_\alpha x$, so that there exists an $i_0$ with $x_{i-p}+q>x_i$ for all $i\leq i_0$. For every $i \in \Z$, there exists an $N\in \N$, such that for all $n>N$, $x_{i-np}>x_{i_0}$, so $(\tau^n_{p,q}x)_i>(\tau^{n-1}_{p,q}x)_i$ since $\tau_{p,q}x>_\alpha x$. This implies that for every $i\in \Z$, $(\tau^n_{p,q}x)_i$ is an eventually increasing sequence. We want to show that this sequence is bounded by $x_i+2$. 

Assume not. Then there is an $n \in \N$ with $np>\tilde{K}$ and an $i \in \Z$ such that $(\tau^n_{p,q}x)_i>x_i+2$. Take a periodic minimizer $y\in \M_{np,nq+1}\subset \mathcal{ABM}$ with $x_i<y_i=(\tau_{np,nq}y)_i+1<(\tau^n_{p,q}x)_i$. Since $\frac{nq+1}{np}>\frac{q}{p}$, it holds that $x>_\alpha y$ and $y>_\omega x$ which implies that the domain of crossing of $x$ and $y$ is larger than $\tilde K$. By the same argument as in the proof of Proposition \ref{rot_number} the domain of crossing is also finite, a contradiction to Theorem A.

Hence, for every $i \in \Z$, the sequence $(\tau^n_{p,q}x)_i$ is eventually increasing and bounded. This means that $\tau_{p,q}x_i-x_i \to 0$ for $i \to -\infty$. But then it holds by Lemma \ref{addendum}, that $\tau_{p,q}x\geq x$, which finishes the proof. An equivalent argument applies to the case $\tau_{p,q}x<_\alpha x$.
\end{proof}

\begin{remark}\label{bound exists}
The proof of Theorem \ref{birkhoff} shows in particular that if $x\in \mathcal{ABM}$ with $\rho(x)=\frac{q}{p}$, and $\tau_{p,q}x>x$, then $x^\pm:=\lim_{n\to \infty}\tau^{\pm n}_{p,q}x$ exists and is $p$-$q$-periodic.
\end{remark}

Theorem \ref{birkhoff} is the first part of the Dichotomy Theorem from Section \ref{outline}. We now elaborate on the second part. The following corollary captures the exponential growth property of non-Birkhoff global minimizers. Recall the definition of the constant $d=6r\lceil 24K^2 (2r+1)r^2\lambda^{-2}\rceil+4r$ from Theorem B.

\begin{corollary}\label{cor1}
Let $x\in \M$ and $d$ as in Theorem B. Assume that  there exist constants $ a,b>0 $ with $0<b<\frac{1}{2d} $ such that $ |x_i| \leq a 2^{b|i|}$ for all $i$. In other words, that $x$ grows slower than exponentially with rate $\frac{1}{2d}$. Then $x\in \BM$.
\end{corollary}

\begin{proof}
If $x$ has smaller than exponential growth with rate $\frac{1}{2d}$, then so do all the translates $\tau_{k,l}x$. Then it holds for every $k,l \in \Z\times \Z$ that also $\tau_{k,l}x-x$ has smaller than exponential growth with constant $\frac{1}{2d}$. This implies that the conditions for Theorem B can not be satisfied, so it follows that $x\in \mathcal{ABM}$. By Theorem \ref{birkhoff}, $x\in \BM$.
\end{proof}

Non-Birkhoff global minimizers, moreover, exhibit an oscillation property described below.

\begin{lemma}\label{non-birkhoff lemma}
Assume that a global minimizer $x\in \M$ is not almost Birkhoff, i.e. $x\notin \mathcal{ABM}$. Then there is a translate $\tilde \tau x \in \{\tau_{ 1, 1}x, \tau_{-1,-1}x\}$ such that the domain of crossing $D$ of $x$ and of $\tilde \tau x$ is infinite. 
\end{lemma}

\begin{proof}
If $x\notin \mathcal{ABM}$ then there exists a translate $\tau_{k,l}x$ such that the domain of crossing for $x$ and $\tau_{k,l}x$ is infinite. By Theorem B, there exist monotone infinite sequences $k_n,l_n\in D$, with $l_n\in [k_n-r,k_n+r]$ and $k_{n+1}-k_n\leq d$, and such that $(x_{k_n}-x_{k_n-k}-l)(x_{l_n-k}+l-x_{l_n})\geq 2^n$ and that $(x_{k_n}-x_{k_n-k}-l)>0$. This implies by Cauchy-Schwartz that there is an infinite subsequence $\{k_{n_j}\}$ of $\{k_n\}$ or $\{l_{n_j}\}$ of $\{l_n\}$, such that $x_{k_{n_j}}-x_{k_{n_j}-k}-l \geq 2^{n/2}$ or $x_{l_{n_j}-k}+l-x_{l_{n_j}}\geq 2^{n/2}$. Assume the first case holds. Then it holds that $x_{k_{n_j}}-x_{k_{n_j}-k}-k>0$ and $x_{k_{n_j}}-x_{k_{n_j}-k}+k>0$, so either $\tau_{k,k}x$ or $\tau_{k,-k}x$ crosses $x$ in an infinite domain (or even $\tau_{k,0}x$ and $x$ cross in an infinite domain). 

Say, $\tau_{k,k}x$ and $x$ cross in an infinite domain $\tilde D$. This implies that $\tau_{k,k}x - x$ changes sign infinitely often in $D$. By writing $$\tau_{k,k}x - x=\tau_{1,1}^kx-\tau_{1,1}^{k-1}x+\tau_{1,1}^{k-1}x\mp...+ \tau_{1,1}x -x ,$$ it is clear that also $\tau_{1,1}x-x$ changes sign infinitely often in some domain $\bar D$. This finishes the proof, where the other case is treated similarly.
\end{proof}

We summarize the results from Theorem \ref{birkhoff}, Corollary \ref{cor1} and Lemma \ref{non-birkhoff lemma} to get the Dichotomy theorem below.
 
\begin{theoremd}
For every global minimizer $x\in \M$ one of the following two cases must hold. 

\begin{itemize}
	\item It holds that $x\in \B$, i.e. $x$ is a Birkhoff global minimizer and thus very regular.
	\item It holds that $x\notin \B$. Then $x$ is very irregular in the following sense. There are monotone infinite sequences $\{k_n, l_n\}\in \Z$, with $|k_{n+1}-k_{n}|\leq d$, $|l_n-k_n|\leq r$ such that one of the following inequalities holds for all $n\in \N$: \begin{align*}
	(x_{k_n+1}-x_{k_n} + 1)(x_{l_n}-x_{l_n+1}+ 1)&\geq 2^n, \text{ or } \\ (x_{k_n+1}-x_{k_n} - 1)(x_{l_n}-x_{l_n+1}- 1)&\geq 2^n.\end{align*} Moreover, for every $n$ at least one of the following must hold: $$x_{k_n+1}-x_{k_n}\geq 2^{n/2}-1, \ \text{ or } \ x_{l_n}-x_{l_n+1}\geq 2^{n/2}-1.$$
\end{itemize}
\end{theoremd}

\begin{proof}
Since $x \notin \BM$, Lemma \ref{non-birkhoff lemma}, gives us a translate $\tilde \tau x \in \{\tau_{ 1, 1}x, \tau_{-1,-1}x\}$, such that the domain of crossing $D$ for $\tilde \tau x$ and $x$ is infinite. By Theorem B there are infinite sequences $\{k_n, l_n\}\in \Z$, with $| k_{n+1}-k_{n}|\leq d$, $| l_n-k_n|\leq r$ and such that $(\tilde \tau x_{k_n}-x_{k_n})>0$ and $(x_{l_n}-\tilde \tau x_{l_n})(\tilde \tau x_{k_n}-x_{k_n})\geq 2^n$. This gives us the first part of the Theorem. 

The second part of the theorem follows by Cauchy-Schwartz.
\end{proof}

This Dichotomy Theorem implies that a global minimizer $x$ that is not Birkhoff has to oscillate in a prescribed uniform way and it has to be growing with some exponential growth rate. Therefore it is very non-physical, as a solution of the generalized Frenkel-Kontorova crystal model.

\appendix
\section{Appendix: Ordering of minimizers}\label{section ordered minimizers}
In Section \ref{section birkhoff} we showed that if a global minimizer is not too wild, it is Birkhoff, i.e. ordered with respect to all its translates. In fact, much more is true. Any Birkhoff global minimizer is ordered with respect to almost all other Birkhoff global minimizers of the same rotation number. We elaborate on this statement below. 

Results in this section follow from the same arguments as in the twist map case (see \cite{MatherForni}). We compare Birkhoff global minimizers of the same rotation number and explain when they are ordered. 

All the proofs in this section hold also for a local energy $S$, satisfying Definition \ref{S}, with the weaker twist condition \begin{equation}\label{weak twist} \p_{j,k}S_i\leq 0, \ \forall j\neq k  \ \text{ and }  \ \p_{i,j}S_i< -\lambda <0, j\in\{i-1,i+1\}.\end{equation} For the sake of bigger generality of the results, we use the weaker twist condition (\ref{weak twist}) in place of the strong twist condition (\ref{strong twist}) used in previous sections because this weaker twist condition has been used in a couple of previous papers (see \cite{llave-lattices}, \cite{ghost_circles}, \cite{destruction}).

We have in mind that one of the following holds. Either the strong twist condition (\ref{strong twist}) holds and the minimizers are known to be in $\mathcal{ABM}$, so they are Birkhoff by Theorem \ref{birkhoff}, or the weaker twist condition (\ref{weak twist}) holds and the minimizers are a-priori known to be Birkhoff.

Since all Birkhoff sequences have a rotation number, we can write the collection of Birkhoff global minimizers as the following union $$\BM:=\bigcup_{\nu \in\R\backslash \Q  }\BM_\nu\cup \bigcup_{q/p \in \Q} \BM_{q/p}^+ \cup \BM_{q/p}^-,$$ defined by $$\BM_\nu:=\{x\in \M\cap \B_\nu\}, \ \text{ for } \ \nu\in \R\backslash \Q$$ and for $\frac{q}{p}\in \Q$, $$\BM_{q/p}^+:=\{x\in \M\cap \B_{q/p}\ | \ \tau_{p,q}x\geq x\} \ \text{ and } \ \BM_{q/p}^-:=\{x\in \M\cap \B_{q/p}\ | \ \tau_{p,q}x\leq x\}.$$ 

The following is a variant of Lemma \ref{strong ordering} that will prove to be useful in the rest of this section and has the same proof. 

\begin{lemma}\label{weak solutions ordering}
Let $x,y$ be solutions to (\ref{rr}) with the weak twist condition, such that $x<y$. Then $x\ll y$.
\end{lemma}

The next lemma is a variant of Lemma \ref{addendum}, but applied to the case of weak twist.

\begin{lemma}\label{addendum2}
Let $x,y \in \M$ be such that $|x_i-y_i|\to 0$ for $i \to - \infty$ and for $i \to +\infty$. Then it holds that $x \ll y$, $x\equiv y$ or $x\gg y$.
\end{lemma}

\begin{proof}
Assume not, so $M=\max\{x,y\}\neq x$ and $m=\min\{x,y\}\neq x$. We claim that $M$ and $m$ are also global minimizers. If $M$ is not, then there is a domain $\tilde B$, a variation $v$ with support in $\mathring{\tilde B}$ and a $\delta>0$, such that for all $B \supset \tilde B$, $W_B(M+v)=W_B(M)-\delta$. 

It holds by (\ref{min-max}) for every $B$ that $W_B(M)+W_B(m)\leq W_B(x)+W_B(y)$. On the other hand, since $x$ and $y$ are asymptotic, there exists for every $\varepsilon>0$ a domain $B_\varepsilon$, such that for all $B\supset B_\varepsilon$ it holds $|W_B(M_B(x))-W_B(M)|\leq \varepsilon$ and $|W_B(m_B(y))-W_B(m)|\leq \varepsilon$. Moreover, by taking $B$ large enough, also $|W_B(M_B(x)+v)-W_B(M+v)|<\varepsilon$ holds. But then for $\varepsilon<\delta/2$ it follows that $W_B(M_B(x+v))+W_B(m_B(y))<W_B(x)+W_B(y)$ which is a contradiction. So it holds by Lemma \ref{weak solutions ordering} that $M\equiv x$ or $M \gg x$ which finishes the proof.
\end{proof}

\subsection{Minimizers of the same irrational rotation number}

Let $\nu \in \R\backslash \Q$ and define the recurrent set of rotation number $\nu$ by $$\mathcal{BM}_\nu^{rec}:=\{x \in \mathcal{BM}_\nu\ | \ x=\lim_{n \to \infty}\tau_{k_n,l_n}x \text{ for some sequences } 0\neq k_n,l_n\}.$$ $\BM_\nu^{rec}$ is also called the Aubry-Mather set of rotation number $\nu$. For the discrete Frenkel-Kontorova model the next theorem was first proved in \cite{AubryLeDaeron} and is explained in \cite{MatherForni}, $\S 12$. A more general version of the proof, applicable to PDEs and monotone variational problems on lattices can be found in \cite{bangert87}. We state it without a proof.

\begin{theorem}\label{uniqueness}
For every $\nu\in \R\backslash \Q$, the recurrent set $\mathcal{BM}_\nu^{rec}$ is the unique smallest nonempty closed subset of $\mathcal{BM}_\nu$ that is invariant under translations.
\end{theorem}

Observe that for any $x \in \mathcal{BM}_\nu$, the $\alpha$- and $\omega$-limit set of the map $\tau_{1,0}:\BM_\nu\to \BM_{\nu}$ defined by $$\alpha(x):=\bigcap_{n\in N}\overline{\bigcup_{l\in \Z}\{\tau_{-1,0}^k(x)+l \ | \ k>n\}} \ \text{ and } \ \omega(x):=\bigcap_{n\in N}\overline{\bigcup_{l\in \Z}\{\tau_{1,0}^k(x)+l \ | \ k>n\}} ,$$ are ordered subsets of $\mathcal{BM}_\nu^{rec}$, because $x$ is Birkhoff. Moreover, by definition they are minimal under translations. So, by the theorem above, the $\alpha$- and $\omega$- limit set for every $x\in \BM_\nu$ are in fact the same set, independent of $x$. This seems at first sight a very surprising result. However, equivalent statements arise in the study of invariant sets of circle homeomorphisms covered by the well known Denjoy theory. Not surprisingly, many proofs in both theories have similar flavors.

Since $\nu$ is irrational, it can be shown that $\mathcal{BM}_\nu^{rec}$ is either homeomorphic to a circle (then it is also called a minimal foliation), or it is a Cantor set (a minimal lamination). Again, this can be explained by a similar argument to the arguments in the Denjoy theory for invariant sets of circle homeomorphisms (for a full proof see e.g. \cite{ghost_circles}, Theorem $4.18$). Theorem \ref{uniqueness} has the following consequence.

\begin{theorem}
For every $\nu\in \R\backslash \Q$, the set of Birkhoff global minimizers of rotation number $\nu$, $\BM_\nu$, is strictly ordered.
\end{theorem}

\begin{proof}
For every $x\in \BM_\nu$, $\alpha(x)$ is ordered with respect to $x$ and by the Theorem \ref{uniqueness}, $\alpha(x)=\mathcal{BM}_\nu^{rec}$. In case $\mathcal{BM}_\nu^{rec}$ is a minimal foliation, we are done because then it holds for every $x\in \mathcal{BM}_\nu$ that $x\in \mathcal{BM}_\nu^{rec}$. In case $\mathcal{BM}_\nu^{rec}$ is a Cantor set, it holds that every gap $[x,y]$ ($(x,y)\cap \mathcal{BM}_\nu^{rec}=\varnothing$) is summable (see e.g. \cite{ghost_circles}, Theorem $10.2$): explicitly, $$\sum_{i\in \Z}y_i-x_i\leq 1.$$ Assume that $z,w\in \mathcal{BM}_\nu\backslash \mathcal{BM}_\nu^{rec}$. Since $\mathcal{BM}_\nu^{rec}=\alpha(z)=\alpha(w)$, $z$ and $w$ have to be ordered with respect to the recurrent set. So, they could cross only if they are in the same gap, but this cannot happen by Lemma \ref{addendum2}. 
\end{proof}

\subsection{Minimizers of the same rational rotation number}

As in the case of twist maps, it holds that for every $\frac{q}{p}\in \Q$, the sets $\mathcal{BM}^+_{q/p}$ and $\mathcal{BM}^-_{q/p}$ are ordered. The arguments are summarized in the following. 

\subsubsection{The periodic case}\label{periodic}
As was explained in the introduction, by definition, $\M_{p,q}$ is the set of $p$-$q$-periodic minimizers that minimize the periodic action $W_{p,q}$. It holds by Aubry's Lemma also for the weaker twist condition (\ref{weak twist}) that $\M_{p,q}\subset \B_{p,q}$ which in particular implies that periodic minimizers are global minimizers. On the other hand, it also holds that every global minimizer which is $p$-$q$-periodic, is a periodic minimizer, in notation $\BM_{q/p}\cap \X_{p,q}=\M_{p,q}$. The proof of these statements can be found in \cite{ghost_circles} as Theorems $4.3$, $4.8$ and $4.9$ and Corollary $4.6$. In particular, $\M_{p,q}$ is ordered.

\subsubsection{Non-periodic rational case}
In this section we show that the sets $\mathcal{BM}^+_{q/p}$ and $\mathcal{BM}^-_{q/p}$ are ordered. We provide the proofs for $\BM_{q/p}^+$, as the other case is analogous. 

Take an arbitrary $x \in \BM_{q/p}^+\backslash \M_{p,q}$. Then for every $i\in \Z$, $(\tau^n_{p,q}x)_i$ is an increasing and bounded sequence and it is clear that $\lim_{n \to \infty}\tau^n_{p,q}x=:x^+ \in \M_{p,q}$ and  $\lim_{n \to \infty}\tau^{-n}_{p,q}x=:x^-\in \M_{p,q}$.  The first step of the proof is to show that there are no periodic minimizers between $x^-$ and $x^+$.

\begin{theorem}\label{rational crossings}
Let $x \in \BM_{q/p}^+\backslash \M_{p,q}$ and $x^-,x^+\in \M_{p,q}$ as defined above. Then there is no $y \in \M_{p,q}$ such that $x^-< y< x^+$.
\end{theorem}

\begin{proof}
Our proof is a variation on a proof in \cite{MatherForni}. Assume the Theorem is not true and that there is such a $y\in \M_{p,q}$. Because stationary points cannot be weakly ordered by Lemma \ref{weak solutions ordering}, it must hold that $x^- \ll y \ll x^+$. Since $x^-,y$ and $x^+$ are periodic, and because $x_i \to x^\pm$ for $i\to \mp \infty$, there is an integer $i_0\in \Z$, such that $x_i>y_i$ for all $i<-i_0$ and $x_i<y_i$ for all $i>i_0$.

For every $B$ it holds by (\ref{min-max}) that $W_B(x)+W_B(y)\geq W_B(m)+W_B(M).$ Let $k$ be such that $kp>2i_0+r$ and look at $\tau_{kp,0}(m)$ which is asymptotic to $m$ and to $x^-$ in $+\infty$. 

Our next claim is that for every $\varepsilon>0$, there exists an $i_\varepsilon$ such that it holds for all $B \supset B_\varepsilon:=[-i_\varepsilon,i_\varepsilon]$ that \begin{equation}\label{estimate1}|W_B(m)-W_B(\tau_{kp,0}m)|\leq \varepsilon.\end{equation} This is true by the following consideration: let $B:=[-i,i]$ and compute $$|W_B(\tau_{kp,0}(m))-W_B(m)|=|W_{B+kp}(m)-W_B(m)|=|W_{[i+1,i+kp]}(m)-W_{[-i+1,-i+kp]}(m)|.$$ If $i>i_0+kp$, then $m\equiv y$ on $[-i,-i+kp]$ and because $x^-$ and $y$ are $p$-$q$-periodic minimizers, it holds that $W_{[-i+1,-i+kp]}(m)=W_{[i+1,i+kp]}(x^-).$ This implies by the equalities above, that $$|W_{B}(\tau_{kp,0}(m))-W_{B}(m)|=|W_{[i+1,i+kp]}(m)-W_{[i+1,i+kp]}(x^-)|.$$ Now it is clear that the claim above holds, since $m_i\to x_i^-$ for $i\to +\infty$. Explicitly, it holds that $|W_{[i+1,i+kp]}(m)-W_{[i+1,i+kp]}(x^-)|\leq L |m_i-x_i^-|$ because of the uniform bound on second derivatives of $S$ and because $|x_i^--x_{i+1}^-|$ and $|m_i-m_{i+1}|$ are uniformly bounded, by the fact that $x^-$ and $m$ are Birkhoff. 

Next, we define the configuration $z$ by $z_i:=M_i$ for $i< i_0$, and $z_i:=m_{i-kp}=(\tau_{kp,0}(m))_i$ for $i \geq i_0$. By definition of $k$ it follows that $\tau_{kp,0}(m) \equiv y$ on $[-i-r,i_0+r]$. Moreover, on $[i_0,i_0+r]$ it holds $z\equiv M\equiv \tau_{kp,0}(m) \equiv y$, so it follows that \begin{align*}W_B(\tau_{kp,0}(m))+W_B(M)=&W_{[-i,i_0-1]}(y)+W_{[i_0,i]}(z)+W_{[-i,i_0-1]}(z)+W_{[i_0,i]}(y)\\ =&W_B(z)+W_B(y) .\end{align*} This equality, together with the minimum-maximum principle and (\ref{estimate1}) gives for all $B\supset B_\varepsilon$, $$W_B(x)+W_B(y)\geq W_B(m)+W_B(M)\geq W_B(y)+W_B(z)-\varepsilon,$$ so \begin{equation}\label{estimate2}W_B(x)+\varepsilon\geq W_B(z).\end{equation}

We claim that $z$ is a global minimizer. Assume not. Then there exists a domain $\bar B$, a variation $v$ with support in $\mathring{\bar B}$ and a $\delta>0$, such that $W_{\bar B}(z) = W_{\bar B}(z+v)+\delta$. Moreover, for all $B\supset \bar B$, it holds that $W_{B}(z) = W_{B}(z+v)+\delta.$ It holds for $z$ that it is asymptotic to $x$ in $+\infty$ and that $z_i=x_i$ for all $i<-i_0$. We change $z$ into variation of $x$ with support in some $\mathring{B}$, by defining $z_B(x)$ where $z_B(x)_i:=z_i$ for all $i\in \mathring{B}$ and $z_B(x)_i:=x_i$ for all $i\notin \mathring{B}$. Since $v$ is supported in $\mathring{\bar B}$ and $\bar B \subset B$, also $z_B(x)+v$ is a variation of $x$. In particular, it also holds \begin{equation}\label{estimate3}W_{B}(z_B(x)) = W_{B}(z_B(x)+v)+\delta.\end{equation} Because $z$ and $x$ are asymptotic and by definition of $B_\varepsilon$, there is a constant $C$, such that \begin{equation}\label{estimate4}|W_B(z)-W_B(z_B(x))|\leq C \varepsilon,\end{equation} for all $B\supset B_\varepsilon$. By choosing $\varepsilon<\delta/(C+1)$ and combining inequalities (\ref{estimate2}), (\ref{estimate3}) and (\ref{estimate4}), we get for all $B$ such that $B_\varepsilon\subset B$ and $\bar B\subset B$, the inequality $$W_B(x)+\delta>W_B(x)+(C+1)\varepsilon \geq W_B(z)+C\varepsilon\geq W_B(z_B(x))=W_B(z_B(x)+v)+\delta.$$ Because $z_B(x)+v$ is a variation of $x$ with support in $\mathring{B}$, this contradicts the assumption that $x$ is a global minimizer, so $z$ must be a global minimizer.

The last part of the proof is to notice that $x$ and $z$ are ordered, but not strictly ordered. Obviously, $x\equiv z$ on $(-\infty,-i_0]$ and $x\leq z$ on $[-i_0,i_0]$, because here $z\equiv M$. On $[i_0,-i_0+kp]$, $z\equiv y$, so by definition of $i_0$, $z>x$. For $i>-i_0+kp$, it either holds $(\tau_{kp,0}m)_i=(\tau_{kp,0}x)_i>x_i$ because $x\in \M_{q/p}^+\backslash \M_{p,q}$, or $(\tau_{kp,0}m)_i=(\tau_{kp,0}y)_i=y_i>x_i$, because $i>i_0$. So $x<z$ but not $x\ll y$, which contradicts Lemma \ref{weak solutions ordering}. This finishes the proof.
\end{proof}

With Theorem \ref{rational crossings}, we can easily get the announced result for this section.

\begin{theorem}
For every $\frac{q}{p}\in \Q$, the sets $\mathcal{BM}^+_{q/p}$ and $\mathcal{BM}^-_{q/p}$ are ordered. 
\end{theorem}

\begin{proof}
Again, we give the proof only for $\mathcal{BM}^+_{q/p}$, as the other case is equivalent. Let $x,y \in \mathcal{BM}^+_{q/p}$. The case where $x,y\in \M_{p,q}$ is covered in section \ref{periodic} and the case for $x\in \BM_{q/p}^+$ and $y\in \M_{p,q}$ is covered in Theorem \ref{rational crossings}. In view of this, let $x,y \in \BM_{q/p}^+\backslash \M_{p,q}$ and look at the ordered periodic minimizers $x^+$ and $x^-$. If $y\notin [x^-,x^+]$, then by Theorem \ref{rational crossings}, it must hold that $y\ll x^-$ so $y\ll x$, or $y\gg x^+$ so $y\gg x$. On the other hand, if $y\in [x^-,x^+]$, then by the same Theorem, $y^+=x^+$ and $y^-=x^-$, so $y$ and $x$ are asymptotic, and by Lemma \ref{addendum2} they are ordered.
\end{proof}

\subsubsection{Heteroclinic connections}

Our last theorem is the equivalent of Theorem $13.5$ from \cite{MatherForni}. It shows that for every gap in the set of periodic minimizers, there are non-periodic global minimizers forming heteroclinic connections between the two periodic minimizers that constitute the gap. 

\begin{theorem}[Heteroclinic connections]\label{heteroclinic connections}
Assume that $x^-,x^+\in \M_{p,q}$ are such that there is no $y\in \M_{p,q}$ with $x^-\ll y \ll x^+$. Then there exist sequences $x\in \BM^+_{q/p}\backslash \M_{p,q}$ and $\bar x \in \BM_{q/p}^-\backslash \M_{p,q}$ such that $$\lim_{n\to \infty}\tau_{p,q}^n x=x^+=\lim_{n\to -\infty}\tau_{p,q}^n \bar x   \ \text{ and } \  \lim_{n\to -\infty}\tau_{p,q}^nx=x^-=\lim_{n\to \infty}\tau_{p,q}^n \bar x.$$
\end{theorem}

\begin{proof}
As throughout this section, we shall prove only the existence of $x\in \BM_{q/p}^+\backslash \M_{p,q}$. Let us take a sequence of rational numbers $\frac{q_n}{p_n} \nearrow \frac{q}{p}$ for $n\to \infty$ and a number $b\in \R$ with $x^-_0<b<x^+_0$. Since $\M_{p_n,q_n}$ is strictly ordered, we may define for every $n \in \N$ the sequence $y^n:=\min\{y\in \M_{p_n,q_n} \ | \ y_0\geq b \}$, so that it follows $y^n_{-p}+q=(\tau_{p,q}y^n)_0< b$. 

Because $\BM_{[q_1/p_1,q/p]}$ is compact and the rotation number is continuous in the topology of point-wise convergence (see \cite{gole01}), there is a convergent subsequence $\{y^{n_k}\}_k$ such that its limit $\lim_{k\to \infty}y^{n_k}=:x\in \BM$ has rotation number $\rho(x)=\frac{q}{p}$. 

By point-wise continuity, it holds that $x_0\geq b$ and $x_{-p}+q \leq b$, so $x_0\geq x_{-p}+q=(\tau_{p,q}x)_0$. This implies by Lemma \ref{weak solutions ordering} that $x\notin \BM_{q/p}^-\backslash \M_{p,q}$ and since there is no $y \in \M_{p,q}$ with $y_0=b$ by assumption, it follow that $x\notin \M_{p,q}$. Hence, $x\in \BM_{q/p}^+\backslash \M_{p,q}$.
\end{proof}

Obviously, the $x$ and $\bar x$ of Theorem \ref{heteroclinic connections} cross, illustrating that $\BM^+_{q/p}\cup \BM^-_{q/p}$ is in general not ordered.

\newpage 
\begin{small}
\bibliographystyle{amsplain}
\bibliography{Aubry-Mather}
\end{small}

\end{document}